\documentclass[11pt,reqno]{amsart}

\usepackage{amsmath,amssymb,amsthm}
\pdfoutput=1
\usepackage{caption}
\usepackage{subfig}
\usepackage{graphicx}
\begin{document}

\newtheorem{lem}{Lemma}[section]
\newtheorem{prop}{Proposition}[section]
\newtheorem{cor}{Corollary}[section]
\numberwithin{equation}{section}
\newtheorem{thm}{Theorem}[section]

\theoremstyle{remark}
\newtheorem{example}{Example}[section]
\newtheorem*{ack}{Acknowledgments}

\theoremstyle{definition}
\newtheorem{definition}{Definition}[section]

\theoremstyle{remark}
\newtheorem*{notation}{Notation}
\theoremstyle{remark}
\newtheorem{remark}{Remark}[section]

\newenvironment{Abstract}
{\begin{center}\textbf{\footnotesize{Abstract}}%
\end{center} \begin{quote}\begin{footnotesize}}
{\end{footnotesize}\end{quote}\bigskip}
\newenvironment{nome}

{\begin{center}\textbf{{}}%
\end{center} \begin{quote}\end{quote}\bigskip}

\newcommand{\triple}[1]{{|\!|\!|#1|\!|\!|}}
\newcommand{\CO}{{\mathbb C}}
\newcommand{\RE}{{\mathbb R}}
\newcommand{\p}{\par\noindent}
\newcommand{\xx}{\langle x\rangle}
\newcommand{\ep}{\varepsilon}
\newcommand{\al}{\alpha}
\newcommand{\de}{\partial}
\newcommand{\la}{\lambda}
\newcommand{\La}{\Lambda}
\newcommand{\ga}{\gamma}
\newcommand{\del}{\delta}
\newcommand{\Del}{\Delta}
\newcommand{\sig}{\sigma}
\newcommand{\ome}{\omega}
\newcommand{\om}{\omega}
\newcommand{\Ome}{\Omega}
\newcommand{\Om}{\Omega}
\newcommand{\C}{{\mathbb C}}
\newcommand{\N}{{\mathbb N}}
\newcommand{\Z}{{\mathbb Z}}
\newcommand{\R}{{\mathbb R}}
\newcommand{\Rn}{{\mathbb R}^{n}}
\newcommand{\Rnu}{{\mathbb R}^{n+1}_{+}}
\newcommand{\Cn}{{\mathbb C}^{n}}
\newcommand{\spt}{\,\mathrm{supp}\,}
\newcommand{\Lin}{\mathcal{L}}
\newcommand{\SSS}{\mathcal{S}}
\newcommand{\F}{\mathcal{F}}
\newcommand{\xxi}{\langle\xi\rangle}
\newcommand{\xei}{\langle\xi-\eta\rangle}
\newcommand{\yy}{\langle y\rangle}
\newcommand{\dint}{\int\!\!\int}
\newcommand{\hatp}{\widehat\psi}
\renewcommand{\Re}{\;\mathrm{Re}\;}
\renewcommand{\Im}{\;\mathrm{Im}\;}
\newcommand{\ov}{\overline}
\renewcommand{\a}{\alpha}
\renewcommand{\b}{\beta}
\newcommand{\g}{\gamma}
\newcommand{\e}{\varepsilon}
\newcommand{\s}{\sigma}
\newcommand{\n}{\noindent}
\newcommand{\ve}{\varepsilon}
\newcommand{\be}{\begin{equation}}
\newcommand{\ee}{\end{equation}}
\newcommand{\bs}{\begin{split}}
\newcommand{\ees}{\end{split}}
\newcommand{\f}{\frac}
\title{Constrained energy minimization \\ and \\ ground states for NLS with point defects } 



\author{Riccardo Adami}
\address{Riccardo Adami: Dipartimento di Matematica e Applicazioni,
  Universit\'a di Milano Bicocca, Via Cozzi 53, 20125 Milano \\ and \\
Istituto di Matematica Applicata e Tecnologie Informatiche, Consiglio
Nazionale delle Ricerche, via Ferrata 1, 27100, Pavia, Italy.} 
\email{riccardo.adami@unimib.it}

\author{Diego Noja}
\address{Diego Noja: Dipartimento di Matematica e Applicazioni,
  Universit\'a di Milano Bicocca, Via Cozzi 53, 20125 Milano \\ and \\
Istituto di Matematica Applicata e Tecnologie Informatiche, Consiglio
Nazionale delle Ricerche, via Ferrata 1, 27100, Pavia, Italy.}
\email{diego.noja@unimib.it}

\author{Nicola Visciglia}
\address{Nicola Visciglia: Universit\`a di Pisa, Dipartimento di
  Matematica, Largo B. 
Pontecorvo 5, 56100 Pisa, Italy.}
\email{viscigli@dm.unipi.it}

\maketitle

\begin{abstract}
We investigate the ground states of the one-dimensional nonlinear Schr\"odinger equation
with a defect located at a fixed point. The nonlinearity is focusing and
consists of a subcritical power. The notion of ground state can be
defined in several (often non-equivalent) ways. We define a ground state as a minimizer of the energy
functional among the functions endowed with the same mass. This is the physically meaningful definition in the main fields of application of NLS.
In this context we prove an abstract theorem that revisits the
concentration-compactness method and which is suitable to treat NLS with inhomogeneities. Then we apply it to three models,
describing three different kinds of defect: delta potential, delta
prime interaction, and dipole. In the three cases we explicitly compute
ground states and we show their orbital stability. This problem had been 
already considered for the delta and for the delta prime defect with a different constrained minimization problem, i.e. defining
ground states as the minimizers of the action on the Nehari manifold.
The case of dipole defect is entirely new.


\end{abstract}
\section{Introduction} \label{sec:intro}
Several one-dimensional physical systems are driven by the focusing nonlinear
  Schr\"odinger equation (NLS)
\begin{equation}\label{pointNLS}
i\partial_t v  + H v  =v|v|^{p-1}, \quad  1< p< 5, \quad v : (t,x) \mapsto
v (t,x) \in {\mathbb C}, \quad v (t, \cdot) \in L^2 (\R), 
\end{equation}
where  $H$ is a selfadjoint operator on
$L^2 (\R)$. A first fundamental step in studying the dynamics of this system  
concerns the possible existence and properties of standing waves and, among
them, of the ground states. While the former are defined as stationary
solutions to equation \eqref{pointNLS}, the latter are 
characterized in terms of variational properties. Generalizing the
usual notion of ground state in linear quantum mechanics to nonlinear
systems, one is led to introduce ground states as the minimizers of
the energy among the functions endowed with the same
$L^2$-norm. Indeed, out of the realm of linear quantum mechanics, such
a notion still proves meaningful, as the $L^2$-norm 
often represents some physically relevant quantities, e.g. number of
particles in Bose-Einstein condensates, or power supply
in nonlinear optics propagation, which are two main fields of
application of NLS. 
While the definition above is common not only in the physical but 
also in the mathematical literature, for example in the classical
analysis based on concentration-compactness methods (see
\cite{[C],[CL]} and references therein), in most recent papers dealing
with NLS with inhomogeneities and defects (see e.g. \cite{[FJ], 
  [FOO], [LeCFFKS], [AN2]}) it is preferred to define as ground states
the minimizers of the so-called action functional among the functions
belonging to the natural Nehari manifold associated to the functional. Such a
notion corresponds to a different way of controlling the physical system, and mathematically often proves easier to
handle. 
In the present paper we adopt the former definition and after proving
a general theorem for the ground states of \eqref{pointNLS}, we apply it
to NLS with point inhomogeneities of various types to show existence
and orbital stability of the ground states. Moreover, we give the
explicit expression of the family of the ground states in the
considered examples. 
The common characteristic in these applications is the lack of
symmetry with respect to the standard NLS due to the presence of a defect in the propagating medium. Such a feature has relevant consequences
on the family of stationary states: when the operator $H$ is the one-dimensional 
laplacian, equation \eqref{pointNLS} is invariant under
the action of the Galileo group, and this symmetry leads to
a rich family of
solitary waves, consisting of orbits of the existing symmetries. 
We are interested in situations in which some symmetries
are possibly broken by the operator $H$, but
some of them
survive and give rise to standing
waves.
 More specifically, in the
examples treated in Sections \ref{sec:results}, \ref{sec:delta}, 
\ref{sec:deltaprime}, \ref{sec:tau}, translational
symmetry is lost due to singularities in the elements of the domain of
$H$, but $U(1)$-symmetry is preserved. 
\par\noindent
To cast the
issue in a suitable
generality
we pose, in the same spirit (but in a different situation) of \cite{[BV]}, the following 
family of variational problems
\begin{equation}\label{eq:minimization}
 I(\rho) \ : = \ 
\inf_{\substack{u\in {\mathcal H}\\\|u\|_{L^2(\R)}=\rho}} {\mathcal E}(u) 
\end{equation}
where
\begin{equation} \nonumber
{\mathcal E}(u)=\frac 12 Q(u, u) - \frac 1{p+1}\|u\|_{L^{p+1}(\R)}^{p+1}
\end{equation}
is the energy associated to equation \eqref{pointNLS}, whose value is conserved
by the flow,
and $$Q:{\mathcal H}\times {\mathcal H}\rightarrow \R$$ 
is a non-negative quadratic form on
a Hilbert space ${\mathcal H}$. Of course,  for a concrete dynamics 
like \eqref{pointNLS}, the space ${\mathcal H}$ does not coincide with $L^2 (\R)$, but rather
with the domain of the quadratic form associated to the operator $H$,
which is smaller than $L^2 (\R)$.\p 

\n To the aim of proving our abstract results,
the Hilbert space ${\mathcal H}$ is required to have 
an embedding 
in $L^2(\R)\cap L^{p+1}(\R)$ in which the validity of Gagliardo-Nirenberg type inequalities is assumed, 
as well as a.e. pointwise convergence (up to subsequences) of weakly 
convergent sequences in ${\mathcal H}$. 
The quadratic form must have a splitting property (see \eqref{eq:weaklimit}) 
and a continuity property (see \eqref{eq:weaklimit1}) with respect to weak convergence. 
With these hypotheses, in Theorem \ref{thm:con-comp} we prove  
a variant of the concentration-compactness
method according to which, if non-vanishing of  minimizing sequence $u_n$ 
is guaranteed from the outset, then $u_n$ is compact
in $\mathcal H$.\p 

 The connection of this abstract framework with the equation \eqref{pointNLS} is 
easily established: given the embedding of ${\mathcal H}$ in $L^2(\R)$, and
provided that $Q$ is closed and semibounded,  then  $Q$
is associated to a unique selfadjoint operator $H$, and by Lagrange multiplier 
theorem and standard operator theory, 
the minimizers of \eqref{eq:minimization}  must solve the 
stationary equation
\be\label{eq:stationary}
Hu \ - \ u |u|^{p-1} = - \omega u
\ee
where $\omega$ is a Lagrange multiplier. 
As  in the case of the free laplacian, for a more general $H$ solutions 
to \eqref{eq:stationary} exist in
$L^2(\R)$ only for $\om$ in a suitable range, giving rise to a branch of stationary solutions;
moreover, the corresponding function $v(x,t)=e^{-i\omega t} u(x)$ is a standing wave solution to \eqref{pointNLS}. 
This standing wave, being a solution of the minimum problem \eqref{eq:minimization}, is a ground state, and, thanks to a classical argument (see \cite{[C],[CL]}), is moreover orbitally stable.\p
Our main concern in the application of this abstract result is the 
case in which the quadratic form $Q$ describes a so-called {\it point
  interaction} (\cite{[AGHH],[ABD]}),  
that is a singular perturbation at 
a point 
of the one-dimensional laplacian. 

\noindent
 A summary of the basic definitions and of the main results on point
  interactions is provided in 
Section \ref{sec:pi}. Here, for the convenience of the reader, we
limit to a general description. 
Let us consider the closed symmetric laplacian on the domain 
$C^{\infty}_0(\R\setminus \{0\})$. On such a domain the laplacian has 
deficiency indices $(2,2)$ and owing to the Von
Neumann-Krein theory it has a  
 four-parameter family of selfadjoint extensions, called point interactions. 
The elements in the domain of these operators are characterized
 by suitable bilateral boundary conditions at the singularity (see formula
\eqref{hamiltonian}), 
while the action coincides with the laplacian out of the singularity. 
The most popular point interaction is the $\delta$ interaction, 
more often called in the physical literature $\delta$ potential or $\delta$ defect, 
defined by the well-known 
boundary conditions \eqref{deltaBC}.\p 
We interpret, quite generally, singular perturbations $H$ of 
the one-dimensional laplacian as describing models of strongly
localized, ideally pointlike, defect or inhomogeneity in the bulk of the medium in which NLS propagation occurs.  
The interactions between field and defect 
are of importance in the study of one-dimensional evolution of Bose-Einstein 
(``cigar-shaped") condensates or the propagation of laser pulses
in a nonlinear Kerr medium.  
In the physical literature, standing waves of NLS with a defect 
are often considered for the relevant cubic case ($p=3$) and in this context they are called 
defect or pinned modes. They are studied, to the knowledge of the authors, 
in the special model case of $\delta$ potential only (see
\cite{[CM],[WMK],[ANS]} and references therein). \par
It is an interesting fact that, beside this analytical and numerical work, recently has been 
experimentally demonstrated the relevant physical phenomenon of trapping of optical solitons 
in correspondence of a defect (a localized photonic potential), present (or put) in the 
nonlinear medium (\cite{Linzon}). 

Rigorous studies of NLS in the presence of impurities described by point interactions  
have been given along several lines, still with an almost exclusive treatment of $\delta$ potential.
The focus of the currently active mathematical research is on orbital stability of standing waves for subcritical NLS with a $\delta$ potential 
(\cite{[FOO],[FJ],[LeCFFKS],{[ACFN2]}}) and $\delta'$ interaction (\cite{[AN2]}), scattering properties of asymptotically solitary solutions of cubic NLS with 
a $\delta$ potential (\cite{[HMZ],[DH]}) with generalization to the case of star graphs (\cite{[ACFN]}), and breathing in nonlinear relaxation (\cite{[HZ]}); finally, a thorough analysis by means of inverse scattering methods for a cubic NLS with $\delta$ potential and even initial data, with results on asymptotic stability of solutions, is given in \cite{[DP]}. Concerning more general issues, in \cite{[AN1]} the well-posedness of the dynamics is proved for the whole family of
point interactions in the cubic case. 
More relevant to the issue of the present paper is the content of
\cite{[GHW]}, where a variational characterization  
of standing waves of NLS with a $\delta$ potential which is similar to
ours in spirit is stated without proof. Here we treat in detail the
case of $\delta$ potential, filling the gap in \cite{[GHW]}, and also the more singular 
cases of $\delta'$ interaction and dipole interaction. 
At variance with the $\delta$ defect, whose form domain 
coincides with the Sobolev space $H^1(\R)$, 
the latter have a form domain given by $H^1(\R^-)\oplus H^1(\R^+)$, and 
boundary conditions in the operator domain which allow for discontinuities of the elements of the domain 
at the position of the defect ($\delta'$ interaction, see \eqref{delta'BC}) or in both 
the element of the domain and its derivative (dipole interaction, see \eqref{transmission}). 
In particular, concerning this last example, we stress the fact that 
only very recently it has been recognized that dipole interaction represents the 
singular perturbation of the laplacian  which correctly describes a 
$\delta'$ potential, i.e. the derivative of a $\delta$, in the sense that it 
can be approximated by suitable rescaled potentials which converge in 
distributional sense to a $\delta'$ distribution 
(see \cite{[GH]}, \cite{[Z1]}, \cite{[Z2]} 
and the Appendix I for a brief discussion).\p 



\noindent
We start with the case of the $\delta$ potential, described by Corollary \ref{cor:delta}, 
and, for $1<p<5$ and for every positive fixed mass, we prove minimization of the energy functional and we explicitly 
give the set of the minima and the related orbital stability. 
The same result holds true for the critical case $p=5$ if the mass 
is small enough, however we skip the treatment of this case in order to
shorten the presentation. 
We emphasize again that, also in the case of $\delta$ potential, 
in which the variational setting is milder, the standing waves and their stability 
properties were known, but their present characterization through constrained energy minimization was not. 
In particular, the cited papers \cite{[FOO],[FJ],[LeCFFKS]} treated orbital stability through 
the method due to Weinstein and Grillakis-Shatah-Strauss, 
i.e. constrained linearization (\cite{[W2],[W3],[GSS1],[GSS2]}). 
Corollaries \ref{cor:deltaprime} and \ref{cor:Qtau} give the minimization properties 
and, correspondingly, orbital stability of the set of minima for the $\delta'$ 
interaction and dipole interaction, for which nothing (except the results in
\cite{[AN2]}) had been previously  
studied in the literature. The results are analogous to those known for the $\delta$ case, 
even if the statements and the proofs are more difficult due to the
more complicated structure of the set of minima, which presents a spontaneous symmetry breaking, 
and to the presence of a singularity in the elements of the energy domain. The last treated case is the dipole interaction, for which we give the explicit set of standing waves, that splits in two subfamilies, one composed of orbitally stable
ground states, and the other of excited states. This case is entirely new.

The plan of the paper is the following. 
In Section 2, after a preliminary 
presentation of the variational framework,
the statement of the main general Theorem \ref{thm:con-comp} is given and 
the applications to point interactions are stated.
In Section 3 the main theorem is proved, while the proof of the results 
on variational characterization of ground states for NLS with point interactions are given in Sections 4, 5 and 6. 
Two appendices close the paper. 
Appendix I provides a short review of the theory of point interactions on the line, 
including those not widely known, and of the main properties of their quadratic forms. 
In Appendix II we present, making use of an elementary 
analysis of the Cauchy problem for the stationary NLS 
with power nonlinearity on the halfline, the explicit 
structure of standing waves for NLS with point defects. 
Other cases of point interactions can be treated with the same general method.
\vskip 10pt
\section{An Abstract Result and Applications to NLS with Point Interaction}
\label{sec:results}
The variational problems we are interested in share the following variational 
structure:
\begin{equation}
 I(\rho) \ : = \ \inf_{\substack{u\in {\mathcal H}\\\|u\|_{L^2(\R)}=\rho}} {\mathcal E}(u) 
\end{equation}
where
\begin{equation}\nonumber 
{\mathcal E}(u)=\frac 12 Q(u, u) - \frac 1{p+1}\|u\|_{L^{p+1}(\R)}^{p+1}
\end{equation}
and $$Q:{\mathcal H}\times {\mathcal H}\rightarrow \R$$ 
is a non-negative quadratic form on a
Hilbert space ${\mathcal H}\ .$\p
On the Hilbert space ${\mathcal H}$ we assume the following properties: 
\begin{equation}\label{eq:embed}
{\mathcal H}\subset L^2(\R)\cap L^{p+1}(\R)
\end{equation}
$$\hbox{ and } \exists \hbox{ } C>0, \sigma\in (0, 1) \hbox{ s.t. }
\|u\|_{L^{p+1}(\R)}\leq C \|u\|_{L^2(\R)}^\sigma\|u\|_{\mathcal H}^{1-\sigma};
$$\begin{equation}\label{eq:almevery}
\hbox{ if } u_n\rightharpoonup \bar u \hbox{ in $\mathcal H$,
 then up to subsequences } u_n(x)\rightarrow \bar u(x) \hbox{ a.e. } x\in \R
\end{equation}
\begin{example}\label{1}
In the following sections we deal with three examples of Hilbert spaces satisfying the
previous requirements: they are given by $H^1(\R)$ (associated to the $\delta$ potential), 
$H^1(\R^-)\oplus H^1(\R^+)$ (associated to the $\delta'$ interaction), and $H_\tau=\{ \psi \in H^1(\R^-)\oplus
H^1(\R^+), \psi(0+) = \tau \psi (0-) \}$ (associated to the dipole interaction or $\delta'$ potential).
\end{example}
Concerning the quadratic form $Q(.,.)$, the following assumptions are made:
\begin{equation}\label{eq:weaklimit}
u_n\rightharpoonup \bar u \hbox{ in $\mathcal H$ } \Longrightarrow Q(u_n-\bar u, u_n-\bar u)= Q(u_n, u_n) - Q(\bar u, \bar u)+o(1);
\end{equation}
\begin{equation}\label{eq:weaklimit1}
u_n\rightharpoonup \bar u \hbox{ in $\mathcal H$ and } Q(u_n, u_n)= Q(\bar u, \bar u) +o(1)
\Longrightarrow u_n\rightarrow \bar u \hbox{ in } \mathcal H 
\end{equation}
\begin{example}\label{2}
Every continuous quadratic form satisfies \eqref{eq:weaklimit}. Concerning \eqref{eq:weaklimit1},
it is satisfied by any quadratic form with the following structure:
$$Q(u, u)= \|u\|_{\mathcal H}^2 + K(u,u)$$
where
$K:{\mathcal H}\times {\mathcal H}\rightarrow \R$ is such that:
$$u_n\rightharpoonup \bar u \hbox{ in $\mathcal H$ } \Longrightarrow
K(u_n, u_n)\rightarrow K(\bar u, \bar u)$$
\end{example}
Next we state a general result on the compactness of minimizing sequences 
to the minimization problems \eqref{eq:minimization} under suitable
assumptions on the form $Q(.,.)$. 
\begin{thm}\label{thm:con-comp}
Let $Q$ be a non-negative quadratic form on the Hilbert space $\mathcal H$ 
and assume \eqref{eq:embed}, \eqref{eq:almevery}, \eqref{eq:weaklimit},
\eqref{eq:weaklimit1}. 
Let $u_n\in \mathcal H$ be a minimizing sequence
for $I(\rho)$, i.e. $$\|u_n\|_{L^2(\R)}=\rho
\hbox{ and } \lim_{n\rightarrow \infty}{\mathcal E}(u_n)=I(\rho).$$ 
Assume moreover that:
\begin{equation}\label{eq:weak}
u_n\rightharpoonup \bar u\neq 0 \hbox{ in } \mathcal H;
\end{equation}
\begin{equation}\label{eq:negative}
I(\mu)<0, \hbox{ } \forall \hbox{ } 0<\mu\leq \rho;
\end{equation}
\begin{equation}\label{eq:boun}
\hbox{ for any compact set $K\subset (0, \rho]$  we have }\end{equation}
$$ \sup_{\substack{ 
\{u\in {\mathcal H}|{\mathcal E}(u)<0\\ 
\|u\|_{L^2(\R)}=\mu,
\mu\in K\}}} \|u\|_{\mathcal H}<\infty$$
Then 
$u_n\rightarrow \bar u \hbox{ in } \mathcal H$ and
in particular $\bar u$ is a minimizer for \eqref{eq:minimization}. 
\end{thm}
We 
give some applications of the previous general theorem to deduce the existence and 
the stability 
of standing waves for NLS with singular perturbation of the laplacian
described by point interactions.

\medskip

\n
1.
We begin with the so-called  attractive $\delta$ interaction. In our notation the pertinent NLS is
\begin{equation}\label{NLSdelta}
i\partial_t v + H_\alpha^\delta v=v|v|^{p-1}, 
\end{equation}
where $H_\alpha^\delta$ is the operator on $L^2 (\R)$ defined on the domain
\begin{equation} \nonumber 
D (H_\alpha^\delta) \ : = \ \left\{ u \in H^2(\R \backslash \{ 0 \}), \ 
u (0+) = u (0-), \ u' (0+) - u' (0-) = - \alpha u (0+),\ \alpha>0 \right\},
\end{equation}
and its action reads
\begin{equation} \nonumber 
(H_\alpha^\delta u) (x) \ = \ - u''(x), \qquad x \neq 0.
\end{equation}
The parameter $\alpha$ is interpreted as the strength of the $\delta$ potential 
(see also Appendix I).\p
In order to deduce the existence and stability of standing waves
to \eqref{NLSdelta}, according to a general argument introduced in \cite{[CL]}
it is sufficient to prove the compactness of minimizing sequences to the following
variational problems:
$$I^{{\mathcal E}_{\alpha,p}^\delta}(\rho) \ : = \ 
\inf_{\substack{u\in  H^1(\R)\\\|u\|_{L^2(\R)}=\rho}} {\mathcal
  E}_{\alpha,p}^\delta  
$$
where
$${\mathcal E}_{\alpha,p}^\delta(u)=
\frac 12 \|u'\|_{L^2(\R)}^2 - \frac \alpha 2  |u(0)|^2- \frac 1{p+1}\|u\|_{L^{p+1}(\R)}^{p+1}$$
is the energy associated to \eqref{NLSdelta}.


\noindent
We also denote by ${\mathcal M}^{{\mathcal E}_{\alpha,p}^\delta} (\rho)$ 
the corresponding set of minimizers (provided that they exist).\\
To present our next result we introduce the function 
\be \label{u+delta}
u_{\alpha, p, \om}(x) \ : = \  \left( \frac{\omega (p+1)}{2\cosh^2
    (\frac{p-1}{2} \sqrt \omega  (|x | + \tilde x))} \right)^\frac
1{p-1} 
\ee
where
$\tilde x=\tilde x(\alpha, \omega, p)$ is given by  
\be \label{xtilde}
\tanh \left( \frac{p-1}{2} \sqrt \omega \tilde x \right)=
\frac{\alpha}{2\sqrt \omega},
\ee
and 
the map
$$h_{\alpha,p}:
(\alpha^2/4, \infty) \ni \omega\mapsto \|u_{\alpha,p, \omega}\|_{L^2(\R)}\in (0, \infty)\ .$$
In Corollary \ref{monotonL2delta} we prove by elementary computation that 
$h_{\alpha,p}$ is a monotonically increasing bijection (see also \cite{[FOO]}), and 
in particular it is well defined its inverse function 
\begin{equation} \nonumber 
g_{\alpha,p}:(0, \infty)\rightarrow
  (\alpha^2/ 4, \infty).
\end{equation} 
\begin{cor}\label{cor:delta}
Let $\alpha>0$, $1<p<5$ and $\rho>0$ be fixed. Let $u_n\in H^1(\R)$ be
a minimizing sequence for $I^{{\mathcal E}_{\alpha,p}^\delta}(\rho)$, i.e.
\begin{equation} \nonumber
\|u_n\|_{L^2(\R)}=\rho \hbox{ and } \lim_{n\rightarrow
  \infty}{\mathcal E}_{\alpha,p}^\delta(u_n)= I^{{\mathcal
    E}_{\alpha,p}^\delta}(\rho) 
\end{equation}
Then \begin{itemize}
\item a) the sequence $u_n$ is compact in $H^1(\R)$;
\item b) the set of minima is given by 
$${\mathcal M}^{{\mathcal E}_{\alpha,p}^\delta} (\rho)=\left\{
e^{i\gamma}u_{\alpha, p, g_{\alpha,p}(\rho)}, \forall \gamma \in \R \right\};$$
\item c) for every $\rho>0$ the set ${\mathcal M}^{{\mathcal E}_{\alpha,p}^\delta} (\rho)$
is orbitally stable under the flow associated to \eqref{NLSdelta}. 
\end{itemize}
\end{cor}
\p

\medskip

\noindent
2.
An analogous result holds true for the case of a nonlinear Schr\"odinger equation 
with an attractive $\delta'$ interaction (see Appendix I) described by the  equation
\begin{equation}\label{NLSdeltaprime}
i\partial_t v + H_\beta^{\delta'} v=v|v|^{p-1},
\end{equation}
 where the operator $H_\beta^{\delta'}$ is defined by
\begin{equation} \begin{split} \nonumber 
D (H_\beta^{\delta'}) & \ : = \ \left\{ u \in H^2(\R \backslash \{ 0 \}), \ 
u' (0+) = u' (0-), \ u (0+) - u (0-) = - \beta u' (0+),\ \beta>0 \right\}, \\
H_\beta^{\delta'} u (x) & \ = \ - u''(x), \qquad x \neq 0.
\end{split} \end{equation}
The operator $H_\beta^{\delta'}$ is selfadjoint on $L^2 (\R)$.

In analogy with the case of the $\delta$ interaction, we are interested in 
the associated minimization problem:
\be \label{funcdeltaprime}
I^{{\mathcal E}_{\beta,p}^{\delta'}}(\rho) \ : = \
\inf_{\substack{u\in  H^1(\R^-)\oplus H^1(\R^+) \\ 
\|u\|_{L^2(\R)}=\rho}} {\mathcal E}_{\beta,p}^{\delta'}
\ee
where
$${\mathcal E}_{\beta,p}^{\delta'}(u)=
\frac 12 \left( \|u'\|_{L^2(\R^-)}^2 +  \|u'\|_{L^2(\R^+)}^2 \right) - \frac 1{2\beta } |u(0+) - u(0-)|^2- 
\frac 1{p+1}\|u\|_{L^{p+1}(\R)}^{p+1}.$$
We stress that in the previous definition, we denoted 
\be \nonumber 
\|u'\|_{L^2(\R^-)}^2 +  \|u'\|_{L^2(\R^+)}^2  \ = \ \lim_{\varepsilon \to 0+}
 \left(\int_{-\infty}^{-\varepsilon} |u'(x)| \, dx +
 \int^{\infty}_{\varepsilon} |u'(x)| \, dx \right). 
\ee
Besides, notice that $u(0\pm)$ are well defined due to well-known continuity 
property of functions belonging to
$H^1(\R^-)\oplus H^1(\R^+)$.\\

\noindent
We also denote by ${\mathcal M}^{{\mathcal E}_{\beta,p}^{\delta'}} (\rho)$ 
the corresponding set of minimizers.\par\noindent
Next, to explicitly describe minimizers, we introduce two families of
functions; the members of the first family are {\it odd} on $\R$ and
the members of the second family do not enjoy any symmetry, so we
call them {\it asymmetric}. Explicitly (see Propositions
\ref{deltaprimeodd} and \ref{deltaprimeas}),
\begin{equation} \nonumber 
u_{odd, \beta, p,\omega} (x) = \left(\frac{\omega (p+1)}{2\cosh^2
  (\frac{p-1}{2} \sqrt \omega (|x| +  \bar x))}\right)^\frac 1{p-1} 
\hbox{ for } x \in \R 
\end{equation} 
where $\beta \sqrt \omega \tanh 
 \left( \frac{p-1}{2} \sqrt \omega \bar x \right)=2,$ so $\bar x > 0 ;$
\begin{equation} \nonumber 
u_{as, \beta, p,\omega} (x)= \pm\left( \frac{\omega (p+1)}{2\cosh^2 (\frac{p-1}{2} \sqrt \omega
  (x+x_{\pm} ))} \right)^{\frac 1{p-1}}\ \hbox{ for } x\in
\R^\pm.
\end{equation}
where, for $\omega>\frac{4}{\beta^2} \frac{p+1}{p-1}$, the couple
$(x_{+}, x_{-})$ is the only solution to
the transcendental system \eqref{tsystem} with $x_{-} < 0 < x_{+} < |x_{-} |$.   


\n We need also to define the map
$$h_{\beta,p}:
(4/\beta^2, \infty) \rightarrow (0, \infty)$$
such that
\be \nonumber
h_{\beta,p}(\omega) = \left\{             
\begin{array}{ll}  \|u_{odd,\beta, p,\omega} \|_{L^2(\R)}                
 & \hbox{ for } \omega\in
 (\frac{4}{\beta^2},\frac{4}{\beta^2} \frac{p+1}{p-1} ]\\                    
 \|u_{as,\beta, p, \omega} \|_{L^2(\R)} 
&  \hbox{ for } \omega\in
 (\frac{4}{\beta^2}\frac{p+1}{p-1}, \infty)        
\end{array}       
\right. 
\ee
By Proposition \ref{monotonL2} the function $h_{\beta, p}$ is
continuous, monotonically increasing
and surjective,  hence there exists its inverse function 
\begin{equation} \nonumber 
g_{\beta,p}:(0, \infty)\rightarrow (4/\beta^2, \infty)\end{equation}
Now we can give the statement of the Corollary that embodies the
applications of Theorem \ref{thm:con-comp} to the problem \eqref{funcdeltaprime}.
\begin{cor}\label{cor:deltaprime}
Let $\beta>0$, $1<p<5$ and $\rho>0$ be fixed. Let $u_n\in
H^1(\R^-)\oplus H^1(\R^+)$ be a minimizing sequence for $I^{{\mathcal
    E}_{\beta,p}^{\delta'}}(\rho)$, i.e. 
\begin{equation*}
\|u_n\|_{L^2(\R)}=\rho \hbox{ and } \lim_{n\rightarrow \infty}
{\mathcal E}_{\beta,p}^{\delta'}(u_n)=I^{{\mathcal E}_{\beta,p}^{\delta'}}(\rho)
\end{equation*}
Then, 
\begin{itemize}
\item a) the sequence $u_n$ is compact in $H^1(\R^-)\oplus H^1(\R^+)$;
\item b) the set of minima is given by:
$${\mathcal M}^{{\mathcal E}_{\beta,p}^{\delta'}} (\rho)=\left\{
e^{i\gamma}u_{odd, \beta,p, g_{\beta,p}(\rho)}, \forall \gamma \in \R \right\}
\hbox{ if } g_{\beta,p}(\rho)\in \left( \frac
  4{\beta^2},\frac{4}{\beta^2} \frac{p+1}{p-1}  \right];$$ 
$${\mathcal M}^{{\mathcal E}_{\beta,p}^{\delta'}} (\rho)=\left\{
e^{i\gamma}u_{as, \beta,p, g_{\beta,p}(\rho)}(\pm \cdot), \forall \gamma \in \R \right\}
\hbox{ if } g_{\beta,p}(\rho)\in \left( \frac{4}{\beta^2}
  \frac{p+1}{p-1}, \infty \right);$$
\item c) for every $\rho>0$ the set ${\mathcal M}^{{\mathcal E}_{\beta,p}^{\delta'}} (\rho)$
is orbitally stable under the flow associated to \eqref{NLSdeltaprime}. 
\end{itemize}
\end{cor}

\medskip

\noindent
3. As a last example, we study the nonlinear Schr\"odinger equation 
with a dipole interaction 
\begin{equation}\label{NLStau}
i\partial_t v + H_\tau^{dip} v=v|v|^{p-1},
\end{equation}
where $H_\tau^{dip}$ is the operator defined on the domain
\be \nonumber 
D (H_\tau^{dip}) \ : = \ \{ u \in H^2 (\R \backslash \{ 0 \} ), \
u(0+) = \tau u (0-), \ u'(0-) = \tau u' (0+) \}.
\ee
In analogy with the previous point interactions we are interested 
in the following variational problem:
$$I_\tau^{{\mathcal E}_{p}^{dip}}(\rho) \ : = \ \inf_{\substack{u\in
    {\mathcal H}^\tau \\\|u\|_{L^2(\R)}=\rho}} {\mathcal E}_{p}^{dip}  
$$
where
$${\mathcal E}_{p}^{dip}(u)=
\frac 12 \left( \|u'\|_{L^2(\R^-)}^2 +  \|u'\|_{L^2(\R^+)}^2 \right)
  - \frac 1{p+1}\|u\|_{L^{p+1}(\R)}^{p+1},$$
and 
\begin{equation}\label{Htao}{\mathcal H}^\tau= \{u\in H^1 (\RE^-) \oplus H^1 (\RE^+)| \ u (0+) = 
\tau u (0-)\}\end{equation}
We denote by ${\mathcal M}^{{\mathcal E}_{p}^{dip}}_\tau (\rho)$
the corresponding set of minimizers (provided that they exist).
In order to state our result first we introduce the function
$$
\chi^-_{\tau, p, \omega} 
(x)=
\left(\frac{\omega (p+1)}{2\cosh^2 (\frac{p-1}{2} \sqrt \omega (x -
  \xi_\pm))}\right)^\frac 1{p-1}, \qquad 
\hbox{ for } x \in \R^\pm
$$
where $\xi_\pm=\xi_\pm(\tau, p, \omega)\in \R$ are defined by
\be \nonumber \label{xi+-} 
\begin{split}
\tanh \left(\frac{p-1}{2} \sqrt \omega \xi_+ \right) \ = & \ \sqrt
\frac{1-|\tau|^{p-1}}{1-|\tau|^{p+3}} \\
\tanh \left( \frac{p-1}{2} \sqrt \omega \xi_- \right) \ = & \ |\tau|^2\sqrt
\frac{1-|\tau|^{p-1}} {1-|\tau|^{p+3}}.
\end{split} \ee
By Proposition \ref{L2Qtau}
we get that the map 
$$
h_{\tau,p}:(0,\infty) \ni \omega \rightarrow
\| \chi_{p, \tau, \omega}^- \|_{L^2(\R)} \in (0,\infty)
$$  
is a monotonically increasing bijection with inverse map given by 
\begin{equation} \nonumber 
g_{\tau,p}:(0,\infty) \rightarrow
  (0,\infty)
\end{equation}
\begin{cor}\label{cor:Qtau}
Let $\tau\in \R\setminus\{0,\pm 1\}$, $1<p<5$ and $\rho>0$ be fixed. 
Let $u_n\in {\mathcal H}^\tau$ (see \eqref{Htao})
be a minimizing sequence for $I^{{\mathcal E}_{p}^{dip}}_\tau(\rho)$, i.e.
\begin{equation}\nonumber
\|u_n\|_{L^2(\R)}=\rho
\hbox{ and }
 \lim_{n\rightarrow \infty}{\mathcal E}_{p}^{dip}(u_n)=I^{{\mathcal E}_{p}^{dip}}_\tau(\rho)
\end{equation}
Then \begin{itemize}
\item a) the sequence $u_n$ is compact in $H^1(\R^-)\oplus H^1(\R^+)$;
\item b) the set of minima is given by:
\begin{align*}
{\mathcal M}^{{\mathcal E}_{p}^{dip}}_\tau (\rho)&=\left\{
e^{i\gamma}\chi^-_{\tau, p, g_{\tau, p}(\rho)}, \forall \gamma \in
\R \right\} &\hbox{ if }&&\tau&\in(1,\infty);\\ 
{\mathcal M}^{{\mathcal E}_{p}^{dip}}_\tau (\rho)&=\left\{
e^{i\gamma}\chi^-_{\frac 1 \tau, p, g_{\frac 1\tau, p}(\rho)}(-\cdot),
\forall \gamma \in \R \right\} &\hbox{ if }&&\tau&\in(0,1);\\ 
{\mathcal M}^{{\mathcal E}_{p}^{dip}}_\tau (\rho)&=\left\{
e^{i\gamma} {\rm sign} (\cdot) \chi^-_{p, |\tau|, g_{|\tau|,p}(\rho)},
\forall \gamma \in \R \right\} &\hbox{ if }& &\tau&\in (-\infty,
-1);\\ 
{\mathcal M}^{{\mathcal E}_{p}^{dip}}_\tau (\rho)&=\left\{
e^{i\gamma} {\rm sign}(\cdot) \chi^-_{\frac 1 {|\tau|}, p, g_{\frac
    1{|\tau|}, p}(\rho)}(-\cdot), \forall \gamma \in \R \right\} &\hbox{
if }&&\tau&\in(-1,0) 
\end{align*}
\item c) for every $\rho>0$ the set ${\mathcal M}^{{\mathcal E}_{p}^{dip}}_\tau (\rho)$
is orbitally stable under the flow associated to \eqref{NLStau}. 
\end{itemize}
\end{cor}

\begin{remark}
In Appendix II it is shown that a second family of standing waves, denoted by
$\chi^+_{\tau, p, \omega}$,
exists for NLS with $H_\tau$ point interaction. This explains the
symbol used for the set of
ground states in the previous statements. The energy of the members of
the family $\{\chi^+_{\tau, p, \omega}\}$ is higher than the energy of the
members family $\{\chi^+_{\tau, p, \omega}\}$ when the mass is fixed,
so that they are excited states of the system. \par\noindent Notice that, in the case $\tau=1$, the
space ${\mathcal H}^\tau$ coincides with $H^1(\R)$ and the quadratic
form coincides with the quadratic form of the free laplacian; hence
the corresponding minimization problem (the classical one already
studied in \cite{[CL]}) enjoys translation invariance, and the
compactness of minimizing sequences as stated in Corollary
\ref{cor:Qtau}, point a), cannot be true. Of course, compactness holds true up to
translations.  A similar conclusion applies to the case
$\tau=-1$; indeed, the minimization problem can be reduced to the one
for $\tau=1$ via the map $\mathcal H^{-1}\ni u\rightarrow {\rm
  sign}(\cdot) u\in \mathcal H^{1}$. Hence, also in the case $\tau=-1$ it
is hopeless to prove the strict compactness stated in $a)$.  By
the argument in Section 6, it is possible to prove that
$a)$ is true also for $\tau=0$, i.e. on the right of the origin
Dirichlet and on the left Neumann boundary conditions. In this case
the minimizers (on the constraint $\|u\|_{L^2(\R)}=\rho$) are given by
the following set:
$$\{e^{i\gamma} \varphi_{p, \omega}\chi_{(-\infty,0)}\}$$
where $\varphi_{p, \omega}$ is the one-dimensional soliton function defined in \eqref{solitone} and
$\omega$ is uniquely given by the condition  
$$\| \varphi_{p, \omega}\|_{L^2(\R)}^2=2\rho^2.$$
Moreover, arguing as in \cite{[CL]}, this set of minimizers satisfies $c)$.
\end{remark}

\vskip10pt
\section{Proof of Theorem \ref{thm:con-comp}}

Since now on $\theta$ is defined as follows: $\|\bar u\|_{L^2(\R)}=\theta \in (0, \rho]$,
where $\bar u$ is given in \eqref{eq:weak}. 
\\
{\em First step: if $\theta=\rho$ then the thesis follows} 
\\
\\
If $\theta=\rho$ then we get 
$u_n\rightarrow \bar u$ in $L^2(\R)$.
By \eqref{eq:embed} (since $u_n$ is bounded in $\mathcal H$ by assumptions 
\eqref{eq:negative} and \eqref{eq:boun})
we get 
\begin{equation}\label{eq:potential}
u_n\rightarrow \bar u \hbox{ in } L^{p+1}(\R)
\end{equation}
Moreover by \eqref{eq:weaklimit} and due to the non-negativity of $Q$ we deduce that
$$\liminf_{n\rightarrow \infty} Q(u_n, u_n)\geq Q(\bar u, \bar u).$$
As a consequence
we get
$$\liminf_{n\rightarrow \infty} {\mathcal E}(u_n)\geq 
{\mathcal E}(\bar u)$$
and hence, since $u_n$ is a minimizing sequence and since $\|\bar u\|_{L^2(\R)}=\rho$,
then necessarily 
$\liminf_{n\rightarrow \infty} {\mathcal E}(u_n)= {\mathcal E}(\bar u)$.
Due to \eqref{eq:potential} necessarily 
$$\liminf_{n\rightarrow \infty} Q(u_n,u_n)= Q(\bar u, \bar u)$$ and hence
we conclude by \eqref{eq:weaklimit1}.
\\
\\
{\em Second step: $\mu^{-2}I(\mu)>\rho^{-2}I(\rho)$, $\forall \mu\in (0,\rho)$}
\\
\\
Let 
$v_n\in \mathcal H$ be a minimizing sequence for
$I(\mu)$, then we have the following chain of inequalities
$$I(\rho)\leq {\mathcal E}\left (\frac{\rho}{\mu} v_n\right )
= \left (\frac{\rho}{\mu}\right )^2 \left (\frac 12 Q(v_n, v_n)
- \frac 1{p+1}\left ( \frac{\rho}{\mu}\right)^{p-1}\|v_n\|_{L^{p+1}(\R)}^{p+1} \right)$$
$$= \left (\frac{\rho}{\mu}\right )^2 \left (\frac 12 Q(v_n, v_n) 
- \frac 1{p+1} \|v_n\|_{L^{p+1}(\R)}^{p+1} \right)$$$$
+ \frac 1{p+1} \left(\frac{\rho}{\mu}\right)^2
\left (1- \left ( \frac{\rho}{\mu}\right)^{p-1}\right ) \|v_n\|_{L^{p+1}(\R)}^{p+1}$$
Since $1- \left(\frac{\rho}{\mu}\right)^{p-1}<0$ we can continue the estimate as follows
$$I(\rho)\leq 
\left (\frac{\rho}{\mu}\right )^2 \left (\frac 12 Q(v_n, v_n) 
- \frac 1{p+1} \|v_n\|_{L^{p+1}(\R)}^{p+1} \right)$$$$
+ \frac 1{p+1} \left(\frac{\rho}{\mu}\right)^2
\left (1- \left ( \frac{\rho}{\mu}\right)^{p-1}\right )
\lim \inf_n \|v_n\|_{L^{p+1}(\R)}^{p+1}
$$
\
By recalling that $v_n$ is a minimizing sequence for $I(\mu)$, we can conclude the proof provided that
$\inf_n \|v_n\|_{L^{p+1}(\R)}^{p+1}>0$. Notice that this last
fact follows easily by \eqref{eq:negative} and by recalling that $Q$ is by assumption a
non-negative
quadratic form.
\\
\\
{\em Third step: the function $(0, \rho)\ni \mu\rightarrow I(\mu)$ is continuous} 
\\
\\
We fix 
$\rho_n\in (0, \rho)$ such that $\rho_n \rightarrow \bar \rho\in (0, \rho)$
and let $w_n\in \mathcal H$ be a minimizing sequence for $I(\bar \rho)$. 
Arguing as above we get the following chain of inequalities:
$$I(\rho_n)\leq {\mathcal E}\left (\frac{\rho_n}{\bar \rho}w_{n}
\right )
=
$$
$$= \left (\frac 12 Q(w_n, w_n)
- \frac 1{p+1} \|w_{n}
\|_{L^{p+1}(\R)}^{p+1}\right )$$$$+
\left (\left (\frac{\rho_n}{\bar \rho}\right )^2 -1\right )
\left (\frac 12 Q(w_n, w_n)
- \frac 1{p+1} \|w_n\|_{L^{p+1}(\R)}^{p+1}\right )$$$$
+ \frac 1{p+1} \left(\frac{\rho_n}{\bar \rho}\right)^2
\left (1- \left ( \frac{\rho_n}{\bar \rho}\right)^{p-1}\right )
\|w_n\|_{L^{p+1}(\R)}^{p+1}
$$
Since $\rho_n \rightarrow \bar \rho$
and $\sup_n \|w_n\|_{\mathcal H}<\infty$ 
(this follows by \eqref{eq:boun}) we  get:
$$\limsup_{n\rightarrow \infty} I(\rho_n)\leq I(\bar \rho)$$
(where we have used the fact that $w_n\in \mathcal H$ is a minimizing sequence for $I(\bar \rho)$).\\
To prove the opposite inequality let us fix
$w_n\in \mathcal H$ such that\begin{equation}\label{choice}\|w_n\|_{L^2(\R)}
=\rho_n
\hbox{ and } {\mathcal E}(w_n)<I(\rho_n)+\epsilon_n\end{equation}
with $\epsilon_n\rightarrow 0$ and
\begin{equation}\label{boundunif}
\sup_{n} \|w_n\|_{\mathcal H}<\infty
\end{equation}
(the existence of $\epsilon_n$ and $w_n$ follows by \eqref{eq:negative}
and \eqref{eq:boun}.\\
Next we can argue as above and we get
$$I(\bar \rho)\leq {\mathcal E}\left (\frac{\bar \rho}{\rho_n}w_{n}\right )
$$
$$= \left (\frac 12 Q(w_n, w_n) 
- \frac 1{p+1} \|w_{n}\|_{L^{p+1}(\R)}^{p+1}\right )$$$$+
\left (\left (\frac{\bar \rho}{\rho_n}\right )^2 -1\right )
\left (\frac 12 Q(w_n, w_n)
- \frac 1{p+1} \|w_{n}\|_{L^{p+1}(\R)}^{p+1}\right )$$$$
+ \frac 1{p+1} \left(\frac{\bar \rho}{\rho_n}\right)^2
\left (1- \left ( \frac{\bar \rho}{\rho_n}\right)^{p-1}\right )
\|w_{n}\|_{L^{p+1}(\R)}^{p+1}
$$
By using \eqref{choice}, \eqref{boundunif} 
and the assumption $\rho_n\rightarrow \bar \rho$ we get
$$I(\bar \rho)\leq \liminf_{n\rightarrow \infty} I(\rho_n)$$
\\
{\em Fourth step: $\theta=\rho$} 
\\
\\
We assume by the absurd $\theta\in (0, \rho)$ and get a contradiction
(notice that we excluded the value $\theta=0$ by the assumption \eqref{eq:weak}).
Notice that by definition of weak limit we get
\begin{equation}\label{weakL2}\|u_n - \bar u\|_{L^2(\R)}^2
=\|u_n\|_{L^2(\R)}^2 - \|\bar u\|_{L^2(\R)}^2+o(1)=\rho^2-\theta^2
+o(1)
\end{equation}
Moreover by combining \eqref{eq:weaklimit}
with the Brezis-Lieb
Lemma \cite{[BrL]} (that can be applied thanks to \eqref{eq:embed} 
and \eqref{eq:almevery}) and using \eqref{weakL2} we get
$${\mathcal E}(u_n)={\mathcal E}(u_n-\bar u)
+ {\mathcal E}(\bar u)+o(1)\geq I(\sqrt {\rho^2 - \theta^2}+o(1)) + I(\theta)+o(1)
$$
which implies by the third step above
$$I(\rho) \geq I(\sqrt {\rho^2 - \theta^2}) + I(\theta)$$
Applying the second step of the present proof, first with $\mu =
  \sqrt{\rho^2-\theta^2}$ 
and then with $\mu = \theta$,
$$I(\rho)> \frac{\rho^2-\theta^2}{\rho^2} I(\rho) + \frac{\theta^2}{\rho^2} I(\rho)
=I(\rho)$$
which is absurd.

\vskip10pt

\section{Proof of Corollary \ref{cor:delta}} \label{sec:delta}

The proof of $c)$, i.e. orbital stability of elements in the set of
minima, follows by combining points $a), b)$ and the classical
argument  
by Cazenave and Lions (see \cite{[C]}, \cite{[CL]}). So we focus on
the proof of $a)$ and $b)$.\\ 
Concerning $a)$ notice first that due to the constraint it is equivalent to work
with the following modified minimization problem
\begin{equation}\label{xxxxx}
\inf_{\substack{u\in H^1(\R)\\
\|u\|_{L^2(\R)}=\rho}} 
\tilde {\mathcal E}_{\alpha, p}^\delta(u)
=I^{\tilde {\mathcal E}_{\alpha,p}^\delta}
(\rho)
\end{equation}
where {we introduced the augmented functional}
$$\tilde {\mathcal E}^\delta_{\alpha, p}(u)=
\frac 12  \|u'\|_{L^2(\R)}^2    - \frac \alpha 2
|u(0)|^2 +\frac{\alpha^2} 8 \|u\|_{L^2(\R)}^2 - \frac 1{p+1} \|u\|_{L^{p+1}(\R)
}^{p+1}$$
We also denote by ${\mathcal M}^{\tilde {\mathcal E}^\delta_{\alpha,
    p}}(\rho)$ the corresponding set of minimizers 
(provided that they exist).
We have to check the hypotheses of Theorem \ref{thm:con-comp},
where we fix the following framework: 
$${\mathcal H}=H^1(\R) \hbox{ and }
Q(u, u)=\frac 12    \|u'\|_{L^2(\R)}^2 
- \frac \alpha 2
|u(0)|^2 +\frac{\alpha^2} 8\|u\|_{L^2(\R)}^2.$$
By general results on the spectrum of $\delta$ interactions, one knows that
$Q(u, u)\geq 0$ (see Section \ref{spectra}, in particular inequality
\eqref{defposdelta}).
According to Examples \ref{1} and \ref{2}, and since \eqref{eq:embed}
and \eqref{eq:almevery}  
are trivial in this framework, 
it is sufficient to check the assumptions \eqref{eq:weak}, \eqref{eq:negative},
\eqref{eq:boun}.
More precisely we have to prove that:
\begin{equation}\label{eq:weakconv}
\hbox{ if
$u_n\in H^1(\R)$ is a minimizing sequence for $I^{\tilde {\mathcal
    E}_{\alpha,p}^\delta}(\rho)$,  then } 
u_n\rightharpoonup \bar u\neq 0;
\end{equation}
\begin{equation}\label{eq:neg}
I^{\tilde {\mathcal E}_{\alpha,p}^\delta}(\mu)<0, \forall \mu>0;
\end{equation}
\begin{equation}\label{eq:bound}
\hbox{ for any compact set $K\subset (0, \infty)$  we have }\end{equation}
$$ \sup_{\substack{ 
\{u\in H^1(\R)|\tilde {\mathcal E}_{\alpha, p}^\delta(u)<0\\ 
\|u\|_{L^2(\R)}=\mu, 
\mu\in K\}}}\|u\|_{H^1(\R)}<\infty$$
First we check \eqref{eq:neg}.
Fix 
$\varphi_\alpha(x)= \sqrt{\frac \alpha 2} e^{-\frac{\alpha}2 |x|}$,
then by direct {inspection} we get
$Q(\varphi_{\alpha}, \varphi_{\alpha})=0$ and $\|\varphi_\alpha\|_{L^2(\R)}=1$. As a 
consequence
$$I^{\tilde {\mathcal E}_{\alpha,p}^\delta}(\mu)
\leq \tilde {\mathcal E}^\delta_{\alpha, p}(\mu \varphi_{\alpha})
 = -\frac 1{p+1} \|\mu \varphi_{\alpha}\|_{L^{p+1}(\R)}^{p+1}
\ = \ - \f {\mu^{p+1} \alpha^{\f {p-1} 2}} {2^{\f{p-3} 2} (p+1)^2} \
< \ 0.$$
Next we check \eqref{eq:weakconv}.
It is sufficient to show that, up to {subsequences},
\begin{equation}\label{npzero}
\limsup_{n\rightarrow \infty}|u_n(0)|>0.
\end{equation}
First notice that, up to subsequences,
\begin{equation}\label{imporneg}
\lim_{n\rightarrow \infty} ||u_n(0)|-\sup_{\R} |u_n||=0.
\end{equation}
Indeed, let $x_n\in \R$ be such that
$|u_n(x_n)|=\sup_{\R} |u_n|$ and assume by the absurd that
\begin{equation}\label{liminf}
\limsup_{n\rightarrow \infty}||u_n(0)|- |u_n(x_n)||>0.
\end{equation}
Then
we get
$$\tilde {\mathcal E}^\delta_{\alpha, p}(u_n)
- \tilde {\mathcal E}^\delta_{\alpha, p}(u_n (x+x_n))= \frac \alpha
2(-|u_n(0)|^2 + |u_n(x_n)|^2)$$ 
and hence by \eqref{liminf}
$$\limsup_{n\rightarrow \infty}  \Big (\tilde {\mathcal E}^\delta_{\alpha, p}(u_n)
- \tilde {\mathcal E}^\delta_{\alpha, p}(u_n (x+x_n))\Big )>0$$
which is in contradiction with the fact that $u_n$ is a minimizing sequence for $I^{\tilde {\mathcal E}^\delta_{\alpha, p}}(\rho)$.\\
Next we prove \eqref{npzero}. Assume it is false, then by \eqref{imporneg}
$\limsup_{n\rightarrow \infty}\|u_n\|_{L^\infty(\R)}=0$
and hence (since $\|u_n\|_{L^2(\R)}=\rho$) 
$\limsup_{n\rightarrow \infty}\|u_n\|_{L^{p+1}(\R)}=0$.
In particular we get
$$I^{\tilde {\mathcal E}_{\alpha,p}^\delta}(\rho)=
\lim_{n\rightarrow \infty} \tilde {\mathcal E}_{\alpha, p}^\delta (u_n)
= \lim_{n\rightarrow \infty} Q(u_n, u_n)\geq 0$$
that is in contradiction with \eqref{eq:neg}.\\
Let us verify \eqref{eq:bound}.
We shall exploit the following Gagliardo-Nirenberg inequality:
\begin{equation} \nonumber 
 \|v\|_{L^\infty(\R)}\leq \|v\|_{L^2(\R)}^\frac 12\|v'\|_{L^2(\R)}^\frac 12
\end{equation}
In view of this inequality, 
for any $v\in H^1(\R)$ such that $\|v\|_{L^2(\R)}\in K$ we get:
$$\tilde {\mathcal E}_{\alpha,p}^{\delta}(v)
\geq \frac 12 \|v'\|_{L^2(\R)}^2 
-\frac \alpha 2 \|v\|_{L^\infty(\R)}^2- \frac 1{p+1} \|v\|_{L^2(\R)}^2\|v\|_{L^\infty(\R)}^{p-1}
$$
$$
\geq  \frac 12 \|v'\|_{L^2(\R)}^2  
- \frac \alpha 2 (\sup K) \|v'\|_{L^2(\R)}
-\frac {(\sup K)^{\frac{p+3}2}}{p+1} \|v'\|_{L^2(\R)}^{\frac{p-1}2}
$$
and
in particular
we have the inclusion
$$
\{v\in H^1(\R)| \|v\|_{L^2(\R)}\in K \hbox{  and } \tilde {\mathcal E}_{\alpha,p}^{\delta}(v)<0\}
$$
$$\subset \{v\in H^1(\R)| \|v\|_{L^2(\R)}\in K \hbox{  and } 
$$$$\frac 12 \|v'\|_{L^2(\R)}^2 
- \frac \alpha 2 (\sup K) \|v'\|_{L^2(\R)}
-\frac {(\sup K)^{\frac{p+1}2}}{p+1} \|v'\|_{L^2(\R)}^{\frac{p-1}2}<0\}$$
and hence due to the assumption $\frac{p-1}2<2$ we conclude \eqref{eq:bound}.\\
Next we prove $b)$.
Let us consider first real-valued solutions of the minimum problem \eqref{xxxxx}.
First notice that all real valued minimizers have to solve the ODE
\eqref{systemdelta} with a suitable Lagrange multiplier $\omega \in
\R$. By Proposition \ref{nonexistencedelta} 
necessarily $\omega>\frac{\alpha^2}4$ and by Proposition \ref{delta}
the real-valued minimizers are uniquely described by 
$\pm u_{\alpha, p, \omega}$.\\ 
Now we show that every element $u$ in the set of minima (possibly
complex-valued)  ${\mathcal M}^{\tilde {\mathcal
    E}_{\alpha,p}^\delta}(\rho)$, has necessarily the structure
$u(x)=e^{i\gamma}u_{\alpha, p, g_{\alpha,p}(\rho)}(x)$, for some
$\gamma \in [0, 2\pi)$.  
First we notice that 
\begin{equation}\label{nonzero}|u(x)|>0, \hbox{ } \forall \hbox{ }
  x\in \R\setminus \{0\}\end{equation}  
Indeed, it is immediately seen that, if $u\in {\mathcal
    M}^{\tilde {\mathcal E}_{\alpha,p}^\delta}$, then $|u| \in 
{\mathcal M}^{\tilde {\mathcal E}_{\alpha,p}^\delta}$ too, thus 
by the above argument we get $|u|=u_{\alpha, p, g_{\alpha, p}(\rho)}$, and hence
\eqref{nonzero} follows by the explicit shape of $u_{\alpha, p, g_{\alpha, p}(\rho)}$. 
As a consequence of \eqref{nonzero} we get
$u (x) =
e^{i \gamma (x)} r(x)$ with $r(x)=|u(x)|$, on each halfline with
$\gamma(x)$ and $r(x)$ smooth, and hence one has 
$$ - r'' - 2 i \gamma' r' - r^{p} + (\om +
\gamma'') r \ = \ 0\ $$
(we have used the fact that any minimizer satisfies the Euler-Lagrange
equation with a suitable multiplier $\omega$). 
Since the imaginary part in the l.h.s. must vanish, it must
  be $\gamma' r'=0$. 
On the other hand, by the argument above $r(x)=|u(x)|$ is still a
(real-valued) minimizer of the energy, then 
it is given by $u_{\alpha, p, g_{\alpha,p}(\rho)}(x)$ which is
never locally constant. As a consequence, we have necessarily
$\gamma' = 0$, and hence it is a constant on every connected component
of $\R^-\cup\R^+$, while $r$ is a positive real-valued minimizer.  
So it must be
\be \nonumber 
u(x) = \left\{             
\begin{array}{ll}                  
 e^{i \gamma_1} u_{\alpha, p, g_{\alpha,p}(\rho)}(x) \hbox{ for }
 x<0\\                    
e^{i \gamma_2}  u_{\alpha, p, g_{\alpha,p}(\rho)}(x) \hbox{ for } x>0              
\end{array}       \right. 
\ee

\par\noindent
By continuity at the origin one must have $e^{i \gamma_1}=e^{i
  \gamma_2}\ .$ This ends the proof. 
\vskip10pt
\section{Proof of Corollary \ref{cor:deltaprime}} \label{sec:deltaprime}
The proof of $c)$ follows by $a)$ and $b)$ in conjunction with
the general argument by Cazenave and Lions (see \cite{[C]}, \cite{[CL]}) giving orbital stability of the ground states.
Next we focus on the proof of $a)$. 
Arguing as in the proof of Corollary \ref{cor:delta}
we introduce the augmented minimization problem 
$$
I^{\tilde {\mathcal E}_{\beta,p}^{\delta'}} (\rho) \ : = \ 
\inf_{\substack{u\in H^1(\R^-) \oplus H^1(\R^+) \\
\|u\|_{L^2(\R)}=\rho}} 
\tilde {\mathcal E}_{\beta, p}^{\delta'}(u),
$$
where the augmented energy is
\be
\label{ebnr}\tilde {\mathcal E}^{\delta'}_{\beta, p}(u)=
\frac 12 \left( \|u'\|_{L^2(\R^-)}^2 + \|u'\|_{L^2(\R^+)}^2 \right)- \frac{1}{2\beta} |u(0+)-u(0-)|^2 
+\frac{2}{\beta^2}\|u\|_{L^2(\R)}^2 - \frac 1{p+1}
\|u\|_{L^{p+1}(\R)}^{p+1}.
\ee
We have to check the hypotheses of  Theorem \ref{thm:con-comp}
in the framework 
$${\mathcal H}=H^1(\R^-) \oplus H^1(\R^+),$$$$ 
Q(u, u)=\frac 12 \left( \|u'\|_{L^2(\R^-)}^2 + \|u'\|_{L^2(\R^+)}^2 \right)- \frac{1}{2\beta} |u(0+)-u(0-)|^2 
+\frac{2}{\beta^2}\|u\|_{L^2(\R)}^2. $$
It is well-known that
$Q(u, u)\geq 0$ (see Section \ref{spectra}, in particular inequality
\eqref{defposdeltaprime}). 

According to Examples \ref{1} and \ref{2}, and since \eqref{eq:embed} and \eqref{eq:almevery}
are well-known in this framework,
it is sufficient to check the assumptions \eqref{eq:weak}, \eqref{eq:negative},
\eqref{eq:boun}.
More precisely we have to prove that:
\begin{equation}\label{eq:weakconv'}
\hbox{ if
$u_n\in H^1(\R^-) \oplus H^1(\R^+)$ is a minimizing sequence for $I^{\tilde {\mathcal E}_{\beta,p}^{\delta'}}(\rho)$}
\end{equation}$$
\hbox{ then } u_n\rightharpoonup \bar u\neq 0;
$$
\begin{equation}\label{eq:neg'}
I^{\tilde {\mathcal E}_{\beta,p}^{\delta'}}(\mu)<0, \forall \mu>0;
\end{equation}
\begin{equation}\label{eq:bound'}
\hbox{ for any compact set $K\subset (0, \infty)$  we have }\end{equation}
$$ \sup_{\substack{ 
\{u\in \mathcal H|\tilde {\mathcal E}_{\beta, p}^{\delta'}(u)<0\\ 
\|u\|_{L^2(\R)}=\mu, 
\mu\in K\}}}\|u\|_{\mathcal H}<\infty$$
The proofs of \eqref{eq:neg'} and \eqref{eq:bound'}
are similar to the proofs of \eqref{eq:neg} and \eqref{eq:bound} and we omit the details.\\
We focus on the proof of \eqref{eq:weakconv'}.
First notice that
\begin{equation}\label{confronto}
I^{\tilde {\mathcal E}_{\beta,p}^{\delta'}}
<\inf_{\substack{u\in H^1(\R)\\\|u\|_{L^2(\R)}=\rho}} \frac 12  
\|u'\|_{L^2(\R)}^2 
+ \frac{2}{\beta^2}\|u\|_{L^2(\R)}^2
- \frac 1{p+1}\|u\|_{L^{p+1}(\R)}^{p+1}
\end{equation}
Indeed, let $\varphi\in H^1(\R)$ be the unique even and positive minimizer for the functional in r.h.s.
(it is well-known that it exists, see \cite{[CL]}).
Next we introduce $\tilde \varphi\in H^1(\R^-) \oplus H^1(\R^+)$
defined as follows:
$$
\tilde \varphi(x)=\left\{ \begin{array}{ccc}
\varphi(x) \hbox{ for } x>0\\-\varphi(x) \hbox{ for } x<0
\end{array} \right.$$ 
Then \eqref{confronto} comes by the following computation:
$$
I^{\tilde {\mathcal E}_{\beta,p}^{\delta'}}
\leq \tilde {\mathcal E}_{\beta,p}^{\delta'}(\tilde \varphi)$$
$$= \frac 12 \|\varphi'\|_{L^2(\R)}^2 + \frac{2}{\beta^2}\|\varphi\|_{L^2(\R)}^2
- \frac 1{p+1}\|\varphi\|_{L^{p+1}(\R)}^{p+1} - \frac 2{\beta} |\varphi(0)|^2$$
$$<\frac 12 \|\varphi'\|_{L^2(\R)}^2 + \frac{2}{\beta^2}\|\varphi\|_{L^2(\R)}^2
- \frac 1{p+1}\|\varphi\|_{L^{p+1}(\R)}^{p+1}$$$$=
\inf_{\substack{u\in H^1(\R)\\\|u\|_{L^2(\R)=\rho}}} \frac 12
\|u'\|_{L^2(\R)}^2 + 
\frac{2}{\beta^2}\|u\|_{L^2(\R)}^2
- \frac 1{p+1}\|u\|_{L^{p+1}(\R)}^{p+1}.
$$
Next, notice that \eqref{eq:weakconv'} follows provided that
\begin{equation}\label{absdelta'}
\limsup_{n\rightarrow\infty} |u_n(0+)- u_n(0-)|>0.\end{equation}
If it is false, then we can consider the functions
$$\tilde u_n(x)= \left\{
\begin{array}{c}
x \hbox{ for } x\in (u_n(0-), u_n(0+))\\
u_n(x - u_n(0+) \hbox{ for } x\in (u_n(0+), \infty)\\
u_n(x - u_n(0-) \hbox{ for } x\in (-\infty, u_n(0-))
\end{array} \right.$$ 
In fact the corresponding normalized functions
$w_n=\rho\frac{\tilde u_n(x)}{\|\tilde u_n(x)\|_{L^2(\R)}}$
satisfy (by assuming that \eqref{absdelta'} is false)
\begin{equation}\label{eqlim}\lim_{n\rightarrow \infty}
\tilde {\mathcal E}_{\beta,p}^{\delta'} (u_n) - \tilde {\mathcal E}_{\beta,p}^{\delta'}(w_n)
=0.
\end{equation}
On the other hand, $w_n\in H^1(\R)$ and $\|w_n\|_{L^2(\R)}=\rho$ then
$$ 
\inf_{\substack{u\in H^1(\R)\\\|u\|_{L^2(\R)}=\rho}} \frac 12 \|u'\|_{L^2(\R)}^2 + 
\frac{2}{\beta^2}\|u\|_{L^2(\R)}^2
- \frac 1{p+1}\|u\|_{L^{p+1}(\R)}^{p+1}$$$$\leq \lim_{n\rightarrow \infty} 
\frac 12 \|w_n'\|_{L^2(\R)}^2 + \frac{2}{\beta^2}\|w_n\|_{L^2(\R)}^2
- \frac 1{p+1}\|w_n\|_{L^{p+1}(\R)}^{p+1}.
$$
$$= \lim_{n\rightarrow \infty}\tilde {\mathcal E}_{\beta,p}^{\delta'}(w_n)$$
This fact and \eqref{eqlim} give a contradiction with \eqref{confronto}.\\
Next we focus on $b)$. Arguing as in  $b)$ of Corollary 
\ref{cor:delta} we can reduce to characterize the real-valued minimizers
minimizers $u(x)\in 
{\mathcal M}^{\tilde{\mathcal E}_{\alpha,p}^{\delta'}}$.
Notice that by the Euler-Lagrange multiplier technique we have that
$u(x)$ solves \eqref{delta'ODE} for a suitable $\omega\in \R$.
In particular the fact that $u(0+)\cdot u(0-)<0$ follows by the following remark.
First of all $u(0+)\neq u(0-)$ otherwise $u \in H^1(\R)$ and it would give a contradiction
with \eqref{confronto}. Moreover by looking at the structure of the
functional \eqref{ebnr} we see that 
necessarily $u(0+)\cdot u(0-)<0$ (if not we could replace $u(x)$ by $({\rm sign} x) u(x)$ and to 
contradict the minimality properties of $u$).\\ 
Next notice that by Proposition \ref{nonexistencedeltaprime}
we deduce that necessarily $\omega>\frac 4{\beta^2}$ and by Proposition \ref{deltaprimeas}
$u \in \{\pm u_{odd, \beta, p, \omega} \}$ or 
$u \in \{\pm u_{as, \beta, p, \omega}(\pm \cdot)\}$ for suitable $\omega$. 
By Propositions \ref{deltaprimeas} and \ref{deltaprimeodd} it is easy to deduce that necessarily 
 $u(x)= \pm u_{odd, \beta, p, g_{\beta, p}(\rho)}(x)$
in the case $g_{\beta, p}(\rho)\in \left( \frac 4{\beta^2},
  \frac{4}{\beta^2}\frac{p+1}{p-1} \right)$.
Furthermore, in order to find the
  minimizer with $L^2$-norm equal to $\rho$, we must compare 
${\mathcal E}_{\beta,p}^{\delta'}(u_{as,\beta,p, \omega_1})$ and $ {\mathcal
  E}_{\beta,p}^{\delta'}(u_{odd, \beta,p, \omega_2})$,
where 
$\omega_1$ and $\omega_2$ are uniquely defined by the condition
\begin{equation} \nonumber
\|u_{as, \beta,p,
    \omega_1}\|^2_{L^2(\R)}=\|u_{odd, \beta,p,
    \omega_2}\|^2_{L^2(\R)}=\rho^2.
\end{equation}
This could be done directly, making use of Proposition \ref{monotonL2} in the Appendix II; but we can notice that, if $
{\mathcal E}_{\beta,p}^{\delta'}(u_{as,\beta,p, \omega_1}) \geq {\mathcal
  E}_{\beta,p}^{\delta'}(u_{odd, \beta,p, \omega_2})$, then we would
conclude that $u_{as,\beta,p, \omega_1}$ is stable, contradicting
Theorem 6.11 in \cite{[AN2]}. Then, it must be 
$${\mathcal E}_{\beta,p}^{\delta'}(u_{as,\beta,p, \omega_1}) < {\mathcal
  E}_{\beta,p}^{\delta'}(u_{odd, \beta,p, \omega_2}),$$
so the proof is complete.

\vskip5pt
\par\noindent
We end this section noticing the spontaneous symmetry breaking of the set of ground states for a NLS with $\delta'$ interaction. This phenomenon is studied in detail in \cite{[AN2]}.

\section{Proof of Corollary \ref{cor:Qtau}} \label{sec:tau}

As in the previous cases, the proof of $c)$ follows by combining $a)$
and $b)$ with the general stability argument by Cazenave 
and Lions (see \cite{[C]}, \cite{[CL]}). In order to prove $a)$
we have to check that all the assumptions
of Theorem \ref{thm:con-comp} are satisfied provided that we choose
$\mathcal H$ to be 
\begin{equation}\label{Htau}{\mathcal H}^\tau= \{u\in H^1 (\RE^-) \oplus H^1 (\RE^+)| \ u (0+) = 
\tau u (0-)\}\end{equation}
and
\begin{equation}\label{Etau}{\mathcal E}(u)=\frac 12  \left( \|
  u^\prime \|^2_{L^2 (\R^-)}  + \|
  u^\prime \|^2_{L^2 (\R^+)} \right)
- \frac 1{p+1}\|u\|_{L^{p+1}(\R)}^{p+1}.
\end{equation}
To this end we premise the following lemma.
\begin{lem}\label{lem:par}
For every $\rho>0, \tau \in \R\setminus\{0, \pm 1\}$, we have 
$$I^{{\mathcal E}^{dip}_{p}}_\tau(\rho)<\tilde I(\rho)$$
where
\begin{equation}\label{lem:def}\tilde I(\rho)=\inf_{\tilde S(\rho)} \frac 12  \| u^\prime \|^2_{L^2 (\RE)}  
- \frac 1{p+1}\|u\|_{L^{p+1}(\R)}^{p+1}
\end{equation}
and
\begin{equation} \nonumber
\tilde S(\rho)= \{  u \in H^1 (\RE)| \ u (0) = 0, \|u\|_{L^2(\R)}=\rho \}.
\end{equation}
Moreover 
\begin{equation}\label{segno}
I^{{\mathcal E}^{dip}_{p}}_\tau(\rho)<0.
\end{equation}
\end{lem}
\begin{proof}
We assume for simplicity $\tau>0$, the other cases can be treated in a similar way.
First let us remark that we have the following obvious inequality
\begin{equation}\label{eq:lem23}  {\mathcal E}(\varphi)\leq \tilde I(\rho) \end{equation}
where ${\mathcal E}(u)$ was defined in \eqref{Etau} and
$${\mathcal E}(\varphi)= \inf_{\substack{u\in H^1(\R)\\\|u\|_{L^2(\R)}=\rho}}
{\mathcal E}(u)
\hbox{ and }
\|\varphi\|_{L^2(\R)}=\rho.$$
We recall that the existence of a constrained minimizer $\varphi$ for $\mathcal E$
is proved in \cite{[CL]}. Moreover since now on we shall use without any further comment
the following symmetry property:
$\varphi(x)=\varphi(-x)$.
Next we introduce the functions
$$\varphi_\omega(x)=\omega^\frac 2{p-1} \varphi(\omega x)$$
We choose $\omega_1, \omega_2>0$ such that
\begin{equation} \nonumber 
\varphi_{\omega_1}(0)=\tau \varphi_{\omega_2}(0)
\end{equation}
\begin{equation}\nonumber 
\sum_{i=1}^2\|\varphi_{\omega_i}\|_{L^2(0, \infty)}^2=\rho^2
\end{equation}
Such a choice is possible since the conditions above are equivalent to:
\begin{equation}\label{111}\omega_1^{\frac{2}{p-1}}=\tau 
\omega_2^{\frac{2}{p-1}}\end{equation}
\begin{equation}\label{222}
\omega_2^{\frac{5-p}{p-1}} \left (\int_{0}^\infty |\varphi|^2 dx
+ \tau ^\frac{5-p}{2} 
\int_{0}^\infty |\varphi|^2 dx\right ) =\rho^2\end{equation}
Being $I^{{\mathcal E}^{dip}_{p}}_\tau(\rho)$ an infimum, one has obviously
\begin{equation}\nonumber 
I^{{\mathcal E}^{dip}_{p}}_\tau(\rho)\leq \sum_{i=1}^2
{\mathcal E}_{0}^\infty (\varphi_{\omega_i})
\end{equation}
where
$${\mathcal E}_0^\infty(u)=\frac 12  \| u^\prime \|^2_{L^2 (0, \infty)}  
- \frac 1{p+1}\|u\|_{L^{p+1}(0,\infty)}^{p+1}$$
and hence \begin{equation}\label{333}
I^{{\mathcal E}^{dip}_{p}}_\tau (\rho)\ \leq \ {\mathcal E}_{0}^\infty
(\varphi)
\sum_{i=1}^2 \omega_i^{\frac{p+3}{p-1}}
\end{equation}
By combining 
\eqref{111} and \eqref{333} we get:
\begin{equation} \nonumber
I^{{\mathcal E}^{dip}_{p}}_\tau(\rho) \ \leq \ {\mathcal E}_{0}^\infty
(\varphi))
\omega_2^\frac{p+3}{p-1} \left( 1
+ \tau ^\frac{p+3}{2}
\right) 
\end{equation}
and hence by \eqref{222} we get 
\begin{equation}\label{ruby}
I^{{\mathcal E}^{dip}_{p}}_\tau(\rho) \ \leq \
 {\mathcal E}_{0}^\infty (\varphi) 
\left(\frac{\rho^2}{\int_{0}^\infty |\varphi|^2 dx
+ \tau^\frac{5-p}{2} 
\int_{0}^\infty |\varphi|^2 dx   }\right)^\frac{p+3}{5-p}
\left (
1 + \tau ^\frac{p+3}{2}
\right).
\end{equation}
Next notice that we can conclude by \eqref{eq:lem23} provided that
\begin{equation}\label{nicole}
{\mathcal E}_{0}^\infty (\varphi) \left(\frac{\rho^2}{\int_{0}^\infty |\varphi|^2 dx
+ \tau^\frac{5-p}{2} 
\int_{0}^\infty |\varphi|^2 dx   }\right)^\frac{p+3}{5-p}
\left (
1 + \tau ^\frac{p+3}{2}
\right)\ <\ {\mathcal E}(\varphi)
\end{equation}
that due to the even character of $\varphi$ is equivalent to
$$\left(\frac{2}{1
+ \tau^\frac{5-p}{2} }\right)^\frac{p+3}{5-p}
\left (1 + \tau ^\frac{p+3}{2}\right)>2$$
where we have used that, as it is well-known,
${\mathcal E}(\varphi)<0$.
More precisely
the inequality above is equivalent to
$$\left (\frac 12 +\frac{\tau^\alpha}2 \right )> 
\left (\frac 12+ \frac{\tau^\beta}2\right )^\frac \alpha \beta$$
where 
$\alpha=\frac{p+3}2$ and $\beta=\frac{5-p}2$. \par\noindent 
In turn this inequality follows by
\begin{equation} \nonumber 
\frac{1+\tilde \tau^\gamma}2>\left
    (\frac{1+\tilde \tau}2 \right)^\gamma 
\end{equation}
where $\gamma=\frac \alpha \beta=\frac{p+3}{5-p}\in (1, \infty)$
and $\tilde \tau= \tau^\beta$, that is satisfied
by the convexity of the function
$s\rightarrow |s|^\gamma$ for $\gamma>1$.
Notice that \eqref{segno} follows by 
\eqref{ruby} and \eqref{nicole} and the well-known fact that ${\mathcal E}(\varphi)<0$.

\end{proof}

Next we prove $a)$. Due to Examples \ref{1} and \ref{2}, and since in our specific context
\eqref{eq:embed} and \eqref{eq:almevery} are satisfied, we have to
check that all the remaining assumptions 
of Theorem \ref{thm:con-comp} are satisfied provided that we choose
$\mathcal H$ 
and
${\mathcal E}$ as in \eqref{Htau} and \eqref{Etau}.
Concerning the assumption \eqref{eq:negative} (in our concrete
situation) it follows by Lemma \ref{lem:par}. 
The proof of \eqref{eq:boun} is similar to the corresponding proof in the case
of Corollary \ref{cor:delta}. We then prove \eqref{eq:weak}, i.e.: 
assume $u_n\in S^\tau (\rho)$ where $S^\tau(\rho)=\{u\in {\mathcal
  H}^\tau| \|u\|_{L^2(\R)}=\rho\}$, and  
$$\lim_{n\rightarrow \infty} \frac 12  \Big (\| u_n^\prime \|^2_{L^2 (\RE^-)} + 
\| u_n^\prime \|^2_{L^2 (\RE^+)}\Big ) - 
\frac 1{p+1}\|u_n\|_{L^{p+1}(\R)}^{p+1}=I^{{\mathcal
    E}^{dip}_{p}}_\tau(\rho), $$
then
$$u_n\rightharpoonup \bar u \neq 0 \hbox{ in } \mathcal H^\tau$$
In fact it is sufficient to prove that
\begin{equation} \nonumber 
\liminf_{n\rightarrow \infty} |u_n(0+)|>0.
\end{equation}
If not, then up to subsequences we can assume
$$\liminf_{n\rightarrow \infty}\ \{|u_n(0+)|+|u_n(0-)|\}=0$$
where we have used the fact that $u (0+) = 
\tau u (0-)$ for any $u\in {\mathcal H}^\tau $.
Next we modify $u_n$ in $w_n\in H^1(\R)$ 
in such a way that $w_n(0)=0$, $\|w_n\|_{L^2(\R)}=\rho$ and
${\mathcal E}^{dip}_{p}(w_n)\rightarrow  I^{{\mathcal E}^{dip}_{\tau,p}}(\rho)$.
As a consequence we deduce 
$\tilde I(\rho)\leq I^{{\mathcal E}^{dip}_{p}}_\tau(\rho)$ 
(for the definition of $\tilde I(\rho)$ see \eqref{lem:def}) that is
in contradiction with Lemma \ref{lem:par}.
The sequence $w_n$ is defined as follows
$$w_n=\rho \frac{\tilde u_n}{\|\tilde u_n\|_{L^2(\R)}}$$
where
$$\tilde u_n(x)=\left \{\begin{array} 
{c}
u_n(x-|u_n(0+)|), \hbox{ } \forall x>|u_n(0+)|\\
u_n(x+|u_n(0-)|), \hbox{ } \forall x<-|u_n(0-)|\\
{\rm sign} (u_n(0+))x, \hbox{ } \forall x\in (0, |u_n(0+)|)\\
-{\rm sign}(u_n(0-)) x, \hbox{ } \forall x\in
(-|u_n(0-)|,0) \end{array}. \right .$$
Finally, we prove $b)$.
Arguing as in the proof of $b)$ in Corollary \ref{cor:delta}
we deduce that it is sufficient to characterize the real-valued 
minimizers $u \in {\mathcal M}^{{\mathcal E}_{p}^{dip}}_\tau (\rho)$.
 Any such $u$ must solve the problem
\begin{equation} \nonumber 
\left \{\begin{array}{c} -u''+\omega u=u|u|^{p-1} \hbox{ on }
    \R\setminus \{0\}, u(x) \in \R, \\ 
u\in H^2(\R^-)\oplus H^2(\R^+)\\
u(0+)=\tau u(0-), \tau u'(0+)=u'(0-)
\end{array} \right.
\end{equation}
for a suitable value of the Lagrangian multiplier $\om$.
First we prove that necessarily $\omega>0$. Indeed, 
by the the minimizing property of $u$ we deduce that
the function $\lambda\rightarrow {\mathcal E}_p^{dip}(\sqrt \lambda 
u(\lambda \cdot))$ 
has a minimum at $\lambda=1$ and hence (by elementary computations)
$$\frac d{d\lambda}\Big [\frac 12 \lambda^2 \Big (\|u'\|_{L^2(\R^-)}^2+
\|u'\|_{L^2(\R^+)}^2\Big ) -\frac 1{p+1}\lambda^{\frac{p-1}2}
\|u\|_{L^{p+1}(\R)}^{p+1} \Big ]_{\lambda=1}=0$$
which implies
$$ \|u'\|_{L^2(\R^-)}^2+\|u'\|_{L^2(\R^+)}^2=\frac{p-1}{2(p+1)}\|u\|_{L^{p+1}(\R)}^{p+1}$$
By combining this identity with the following one
$$ \|u'\|_{L^2(\R^-)}^2+\|u'\|_{L^2(\R^+)}^2+\omega \|u\|_{L^2(\R)}^2=\|u\|_{L^{p+1}(\R)}^{p+1}$$
(obtained by multiplication of \eqref{QtauODE} by $u(x)$) we deduce that
$\omega>0$.\\
As a consequence we can apply Proposition \ref{Qtau} and get that $u \in
\{\pm \chi^+_{\tau, p, \omega}, \, \pm \chi^-_{\tau, p, \omega}\}$. 
Notice that by Proposition \ref{L2Qtau} the maps $\omega\rightarrow
\|\chi^\pm_{\tau, p, \omega}\|_{L^2(\R)}$ 
are bijective, hence the proof of $b)$ is complete provided that we show
\begin{equation}\label{encomparison}
{\mathcal E}^{dip}_p(\chi^{-}_{\tau, p, \omega_-})<{\mathcal
  E}^{dip}_p(\chi^{+}_{\tau, p, \omega_+}), 
\end{equation}
where $\omega_\pm$ are selected in such a way that
\begin{equation} \nonumber 
\|\chi^{+}_{\tau, p, \omega_+}\|_{L^2(\R)}^2
=\|\chi^{-}_{\tau, p, \omega_-}\|_{L^2(\R)}^2=\rho^2,
\end{equation}
that, due to Proposition \ref{L2Qtau}, is equivalent to
\begin{equation}\label{ideexpliL2}
\left( \frac {(p+1)}2 \right)^\frac 2{p-1} \frac{2 }{(p-1)}
\left ( \int_{-1}^1 (1-t^2)^\frac{3-p}{p-1}  \pm  \int^{\tau^2 \sqrt 
\frac{1-\tau^{p-1}}{1-\tau^{p+3}}}_{\sqrt
\frac{1-\tau^{p-1}}{1-\tau^{p+3}}} (1-t^2)^\frac{3-p}{p-1} dt \right ) 
\end{equation}
$$
=\rho^2 (\omega_\pm)^\frac{p-5}{2(p-1)}.
$$
In order to perform the comparison, first notice that, being solutions
to \eqref{QtauODE}, the functions $\chi^\pm_{\tau, p, \omega_\pm}$
belong to
the natural Nehari manifold, namely 
$$
\f 1 2 \left( \|(\chi^{\pm}_{\tau, p, \omega_\pm})'\|_{L^2(\R^-)}^2 +
\|(\chi^{\pm}_{\tau, p, \omega_\pm})'\|_{L^2(\R^+)}^2 \right)
+\omega_\pm \|
\chi^{\pm}_{\tau, p, \omega_\pm}\|_{L^{2}(\R)}^{2} 
-\|\chi^{\pm}_{\tau, p, \omega_\pm}\|_{L^{p+1}(\R)}^{p+1}\ =\ 0,
$$
so that their energy can be written as
$$
{\mathcal E}_p^{dip} (\chi^{\pm}_{\tau, p, \omega_\pm}) \ = \ \f {p-1} {2(p+1)} \| 
\chi^{\pm}_{\tau, p, \omega_\pm} \|_{L^{p+1}(\R)}^{p+1} - \f {\omega_\pm} 2  \|\chi^{\pm}_{\tau, p, \omega_\pm} \|_{L^{2}(\R)}^{2}
$$
and hence by Proposition \ref{L2Qtau} we get
$${\mathcal E}_p^{dip} (\chi^{\pm}_{\tau, p, \omega_\pm})=$$$$\omega_\pm^\frac{p+3}{2(p-1)} 
\Big(\frac{p+1}{2} \Big)^\frac 2{p-1}\Big [\frac 12 
\Big (\int^1_{-1} (1-t^2)^\frac{2}{p-1} dt 
\pm  \int_{\sqrt \frac{1-\tau^{p-1}}{1-\tau^{p+3}}}^{\tau^2 \sqrt 
\frac{1-\tau^{p-1}}{1-\tau^{p+3}}} (1-t^2)^\frac{2}{p-1} dt\Big )
$$
$$- \frac 1{p-1} 
\Big (\int^1_{-1} (1-t^2)^\frac{3-p}{p-1} dt 
\pm \int_{\sqrt \frac{1-\tau^{p-1}}{1-\tau^{p+3}}}^{\tau^2 \sqrt \frac{1-\tau^{p-1}}{1-\tau^{p+3}}} 
(1-t^2)^\frac{3-p}{p-1} dt \Big )\Big ]. $$
By the following identity, obtained by integrating by parts 
$$
\int(1-t^2)^{\frac{2}{p-1}}\ =\ \frac{p-1}{p+3} t(1-t^2)^\frac{2}{p-1}
+ \frac{4}{p+3}\int (1-t^2)^{\frac{3-p}{p-1}}  ,
$$
we get
$${\mathcal E}_p^{dip} (\chi^{\pm}_{\tau, p, \omega_\pm})= $$
$$ \omega_\pm^\frac{p+3}{2(p-1)} 
\Big(\frac{p+1}{2} \Big)^\frac 2{p-1}\frac{p-5}{(p-1)(p+3)}
 \Big (\int^1_{-1} (1-t^2)^\frac{3-p}{p-1} dt 
\pm \int_{\sqrt \frac{1-\tau^{p-1}}{1-\tau^{p+3}}}^{\tau^2 \sqrt \frac{1-\tau^{p-1}}{1-\tau^{p+3}}} 
(1-t^2)^\frac{3-p}{p-1} dt\Big )
$$
that in conjunction with \eqref{ideexpliL2} implies
\eqref{encomparison}. The proof is complete.

\vskip10pt

\section{Appendix I: Review of Point Interactions} \label{sec:pi}
In this section we describe all interactions in dimension one that
are concentrated in a single point. 
From a physical point of view these operators (and 
the corresponding quadratic forms) can be interpreted as the
family of hamiltonian operators describing the 
dynamics of a particle in dimension one under the influence of an
impurity, or defect, acting as a capture or scattering 
centre. Placing the origin of the line at the
centre of interaction, one can rigorously obtain such hamiltonian operators as the
set of 
 selfadjoint
extensions (s.a.e.) of the symmetric operator
\begin{equation} \label{h0}
 H_0 u = - u''
\end{equation} 
defined on the domain
\begin{equation} \nonumber 
D ( H_0 ) = C^\infty_0 (\R \backslash \{0\} ),
\end{equation}
i.e. the set of smooth, compactly supported functions that vanish
  in some neighbourhood of the origin.

\p 
By the Krein's theory of s.a.e. for symmetric operators on Hilbert
spaces (see \cite {[AG]}) one easily proves that there is a
$4$-parameter family of s.a.e. of \eqref{h0}. Such a family can be
equivalently described through a $4$-parameter family of boundary
conditions at the origin. 
{Summarizing the results in \cite{[ABD]} and \cite{[EG]},
 the explicit action and domain of the so
constructed operators, following \cite{[AGHH],[AK],[ABD],[BEH]} and
reference therein,  
can be conveniently given by distinguishing two families of s.a.e.

\vskip10pt
\p
{\it Coupling point interactions}: given $\ome\in \CO, a, b, c, d \in \RE$ such that $| \ome |  = 1, \ ad-bc = 1$,  we define the s.a.e. $H_U$ as follows:

\begin{equation} \begin{split} \label{hamiltonian}
U & = \ \ome \left( \begin{array}{cc} a & b \\ c & d \end{array}
\right),  \\
D_U \ : = \ D (H_U) & = \ \left\{ u \in H^2 (\RE \backslash \{0\}),  \
\left( \begin{array}{c} u(0+) \\   u^\prime (0+) 
\end{array} \right)  = U 
\left( \begin{array}{c} u(0-) \\   u^\prime (0-) 
\end{array} \right) \right\}, \\
(H_U u) (x) & = \ - u^{\prime\prime} (x), 
\qquad x \neq 0, \qquad 
\forall \, u  \in \ D (H_U)
\end{split} 
\end{equation} 
We stress that the dynamics generated by any Hamiltonian $H_U$ couples the negative
real halfline with the positive one. In other words, if a wave packet initially confined in the
negative halfline is acted on by a linear Schr\"odinger or heat
dynamics, it instantaneously diffuses in the positive halfline, and vice
versa. In the case of the linear wave equation, there is equally propagation through
the interaction centre, but at a finite velocity. This is why members of
this class 
of point interactions are called {\em coupling}. 
\vskip10pt
\p
{\it Separating point interactions}: given $p, q \in \RE \cup \{ \infty \}$ we define the s.a.e. $H_{p,q}$
as follows:
\begin{equation}\label{hamil2}
D_{p,q} \ : =  \ D (H_{p,q}) 
\end{equation}
$$= \   \left\{ u \in H^2 
(\RE \backslash \{0\}), 
\ u (0+) = p u^\prime (0+), \ \ u (0-) = q u^\prime (0-)
\right\} , $$$$
 (H_{p,q} u) (x) \ =   \ -u^{\prime\prime} (x), 
\qquad x \neq 0 \qquad 
\forall \, u  \in \ D (H_{p,q})
$$
\p
These boundary conditions are opaque to transmission of the
wavefunction from one half axis to the other,  
and allow just reflection, with Robin boundary conditions on the two sides. 
In particular, the cases of right, left or bilateral Neumann or
Dirichlet boundary conditions are found by choosing 
  $p = \infty$ or $p = 0$ in $0+$, and/or
  $q = \infty$ or $q = 0$ in $0-$.

\vskip5pt
\p
Notice that by choosing in the matrix $U$ the coefficients $\omega = a
= d = 1$,  $b = c = 0$, one obtains the free-particle
Hamiltonian $H_0 u= - u''$ on its standard domain $H^2(\RE)$.
\p
\p

\p
Non-trivial examples are the following.\p
\vskip5pt
The choice $\ome = a = d = 1$, $b = 0$, $c =- \alpha \neq
0$
corresponds to the well-known case of a pure Dirac $\delta$
interaction of strength $- \alpha$, from now on noted as $H_{\alpha}^\delta$. \p
We note explicitly that our sign convention on the strength is
different from the usual one (which correspond to the exchange
$\alpha\rightarrow -\alpha$), because in the present paper we are
interested in the delta potential with just one sign of $\alpha$, the
one which corresponds to attractive interaction, and we want to keep
it positive along the analysis.\p 
Explicitly,
\begin{equation} \label{deltaBC} 
\begin{split}
u (0+) & = \ u (0-) \\
u^\prime(0+) - u^\prime (0- ) & = \ -  \alpha  u(0-)
\end{split}
\end{equation}
\p
The $\delta$ interaction $H_{\alpha}^\delta$ is the norm-resolvent
limit of a family of Schr\"odinger operators $H_{\epsilon,
  \alpha}u=-u''- \alpha\frac{1}{\epsilon}V(\frac{x}{\epsilon})u$ with
$\int_{\RE}V(s)ds=1$. The family
$\frac{1}{\epsilon}V(\frac{x}{\epsilon})\rightharpoonup\delta_0$
  in
distributional sense as $\epsilon\to 0$. This justifies the name of
$\delta$ potential.
\vskip5pt
The case $\ome = a =
d = 1$, $c = 0$, $b = - \beta \in \RE$ corresponds to the case of the
so-called $\delta^\prime$ interaction of strength $- \beta$. 
To be explicit, the boundary conditions are
\begin{equation}\label{delta'BC} \begin{split}
u^\prime (0+) & = \ u^\prime (0-) \\
u (0+) - u (0- ) & = \ - \beta  u^\prime (0-)
\end{split}
\end{equation}
Note that in the $\delta$ interaction the functions in the domain are
continuous and their derivatives have a jump at the origin, while in
the $\delta'$ case the functions have a jump at the origin, and their
left and right derivatives coincide.\p  The same remark on sign convention made for the $\delta$ potential applies to the $\delta'$ interaction: the usual one corresponds to the exchange $\beta\rightarrow -\beta$, and we use the present one because we are interested just in one sign of $\beta$, the one which corresponds to attractive interaction, and we want to keep it positive.
It has been proven that the $\delta'$ interaction does not correspond
to the norm-resolvent limit of a family of Schr\"odinger operators
with potentials $\frac{1}{\epsilon^2}V(\frac{x}{\epsilon})$
approximating the $\delta'$ distribution in the limit $\epsilon\to 0$ (i.e. $\int_{\RE}\ V(s)ds=0$ and
$\int_{\RE}\ sV(s)ds=-1$) . It is, in fact, the norm-resolvent limit
of a more complicated family
of Sch\"odinger  operators, a subject of some concern in the literature (see \cite{cheon,[ENZ]} and reference therein).
So, the question arises of which boundary condition or point interaction, if any exists, describes a $\delta'$ potential, in the sense stated.
Let us consider the interaction $H_{\tau}$ given by the following transmission
boundary conditions for  
$ u \in H^2 (\RE \backslash \{0\})$,
\begin{equation}\label{transmission} \begin{split}
u^\prime (0-) & = \ \tau u^\prime (0+) \\
u (0+) & = \ \tau u (0-)
\end{split}
\end{equation}
and action $H_{\tau} u =-u''$.\par\noindent
It has been recently shown (see \cite{[GH]}) that these boundary
conditions describe the norm-resolvent limit  
of the family of s.a. Schr\"odinger operators $H_{\epsilon}^{\gamma}
u=-u''+\gamma \frac{1}{\epsilon^2}V(\frac{x}{\epsilon})u$  
with $\int_{\RE}V(s)ds=0\ $ and  $\int_{\RE}\ sV(s)ds=-1 $, when a
suitable resonance condition on the
potential $\gamma V$ is satisfied; moreover the parameter 
$\tau$ emerges as a scalar function of the resonance of
$\gamma V$.\par\noindent
Precisely, if the potential $\gamma V $ has a zero energy resonance
with resonance function $u_{\gamma}$ (i.e. a $L^\infty$ solution of
$-u_{\gamma}^{''} + \gamma V u_{\gamma}= 0 $ with existing $u_{\gamma}
(\pm\infty)=u_{\gamma}^{\pm}$), then the norm-resolvent limit of the
operator $H_{\epsilon}^{\gamma}$ coincides with the operator
$H_{\tau}$ where
$\tau=\tau_{\gamma}=\frac{u_{\gamma}^{+}}{u_{\gamma}^{-}}$. 
 On the contrary,  in the non-resonant case the scaled
Schr\"odinger operator converges to $ - \frac{d^2}{dx^2}$  
with bilateral Dirichlet boundary conditions, which is a separating trivial case. 
This fact strongly suggests to consider the boundary conditions defining $H_{\tau}$ as describing a 
$\delta'$-potential or in physical terms a {\it dipole}
interaction. We emphasize again that the norm-resolvent limit yielding
$H_{\tau}$ depends on the regularization, i.e. depends on the shape
(through its resonances) of the potential $\gamma V$ approximating in
distributional sense $\gamma \delta'$. 
This feature is at variance with the case of a $\delta$ interaction, which is a a
norm-resolvent limit of a family of regular potentials {\it
  independent of the
  regularization}. 

\p
We finally mention that a wide set of point interactions can be
recovered as the limit case of a Schr\"odinger operator on a line with
a junction of finite width $[-\Lambda, \Lambda]$ and suitable boundary
conditions in $\pm\Lambda$, in the limit of vanishing 
$\Lambda$. See \cite{[FHNS]} for details on this model
and for an
interesting physical interpretation. 

\vskip10pt
\par\noindent
Now we discuss the quadratic form associated to the point interactions previously defined.\p 
We recall (for details see e.g. \cite{[RS1]}) that the
quadratic form $Q_{A}$ associated to a selfadjoint operator $A$ is the
closure (ever existing) of the quadratic form given by $q_{A}(u,u)=(u,
Au) $, for $ u \in D(A)$ and denoted by $(,)$ the inner product of the
underlying Hilbert space. The form domain $D(Q_A)$ of the closure
turns out to be an extension of the operator domain $D(A)$. The form
$Q_{A}$ has often the meaning of energy, and the form domain $D(Q_A)$
that of domain of the finite energy states. Here we adopt this
usage. Moreover, in the following we omit the subscript $A$ that refers to
the original s.a. operator, in favour of a more agile notation. No
ambiguity should be present. 

\vskip5pt
\par\noindent
The quadratic forms associated to point interactions are defined as
follows. 
\vskip5pt
\noindent
1.
For the Hamiltonian $H_{0,0}$ corresponding to bilateral Dirichlet b.c.
the energy space is
\begin{equation} \nonumber 
D(Q_0) : = \{ u \in H^1 (\RE), \ u (0) = 0 \}
\end{equation}
and the form reads
\begin{equation} \nonumber 
Q_0 (u) \ = \ \| u^\prime \|^2.
\end{equation}

\noindent
2.
For the Hamiltonian $H_{0,q}$, $q \neq 0$ (right Dirichlet b.c.)
\begin{equation} \nonumber
D(Q_{0+}): = \{ u \in H^1 (\RE^-) \oplus H^1 (\RE^+), \ u (0+) = 0 \}
\end{equation}
and 
\begin{equation} \nonumber
Q_{0+} (u) \ = \ \| u^\prime \|^2_{L^2 (\RE^-)} + 
\| u^\prime \|^2_{L^2 (\RE^+)} - | q |^{-1} | u (0-) |^2.
\end{equation}
\vskip5pt
\noindent
Analogously (left Dirichlet b.c)
\begin{equation} \nonumber
D(Q_{0-}): = \{ u \in H^1 (\RE^-) \oplus H^1 (\RE^+), 
\ u (0-) = 0 \}
\end{equation}
and the form reads
 \begin{equation}\nonumber
Q_{0-} (u) \ = \ \| u^\prime \|^2_{L^2 (\RE^-)} + 
\| u^\prime \|^2_{L^2 (\RE^+)} + | p |^{-1} | u (0+) |^2.
\end{equation}
\noindent
3.
For the Hamiltonian $H_U$, defined in \eqref{hamiltonian}, with
 $b = 0$ the energy space is
\begin{equation}\nonumber
D(Q_{\omega a}): = \{  u \in H^1 (\RE^-) \oplus H^1 (\RE^+), \ u (0+) = 
\omega a u (0-) \}
\end{equation}
and the form reads
\begin{equation} \nonumber
Q_{\ome a}  (u) \ = \  \| u^\prime \|^2_{L^2 (\RE^-)} + 
\| u^\prime \|^2_{L^2 (\RE^+)}  + ac
|u (0-) |^2.
\end{equation}

\noindent
4.
For any other s.a.e. of $H_0$ the energy space is given by
 \begin{equation}\nonumber
D(Q) : =  H^1 (\RE^-) \oplus H^1 (\RE^+)
\end{equation}
To describe the action of the form we have to
consider two cases:
\vskip5pt
\noindent
4.a.
if the Hamiltonian is of the type $H_U$ described in
  \eqref{hamiltonian}, with $b \neq 0$, then
\begin{equation*}
Q_U \, : = \,  \| u^\prime \|^2_{L^2 (\RE^-)} + 
\| u^\prime \|^2_{L^2 (\RE^+)}$$$$
 + b^{-1}[ d| u (0+)
|^2 + a | u (0-) |^2 - 2 {\rm{Re}} ( \ome \ov{u (0+)} u (0-)) ].  
\end{equation*}
\noindent
4.b.
if the Hamiltonian is of the type $H_{p,q}$ described in 
\eqref{hamil2}, with $p,q$ both different from zero, then 
\begin{equation} \nonumber
Q_{p,q} (u) \ : = \   \| u^\prime \|^2_{L^2 (\RE^-)} + 
\| u^\prime \|^2_{L^2 (\RE^+)} + p^{-1} | u (0+) |^2 - q^{-1}
| u (0-) |^2.
\end{equation} 

\vskip10pt
\p
All above energy spaces can be endowed with the structure
of Hilbert space by introducing the  hermitian product
\begin{equation} \nonumber 
(u, \phi)_X \ = \ ( u, \phi) + 
 \lim_{\ve \rightarrow
  0+}  \int_\ve^{+ \infty} \ov{u^\prime (x)} \phi^\prime (x) \, dx + 
 \lim_{\ve \rightarrow
  0+}  \int^{-\ve}_{- \infty} \ov{u^\prime (x)} \phi^\prime (x) dx.
\end{equation}

We give more explicitly the quadratic forms and their domains
for the examples of $\delta $ interaction  
$H_{\alpha}^\delta$, 
$\delta'$ interaction 
$H_{\beta}^{\delta'}$ and $\delta'$ potential $H_{\tau}^{dip}$.
\vskip10pt\p
For the $\delta$ interaction with $\alpha \neq 0$ we have 
\begin{equation} \nonumber
D(Q_\alpha^\delta)=H^1 (\RE),\quad \quad 
Q_{\alpha}^\delta (u) \ = \ \| u^\prime \|^2_{L^2(\R)} - \alpha | u (0) |^2\ 
\end{equation}
For the $\delta'$ interaction with $\beta \neq 0$:
 \begin{equation}
\nonumber 
D(Q_{\beta}^{\delta'})=H^1 (\RE^-) 
\oplus H^1 (\RE^+),\quad
Q_{\beta}^{\delta'} (u) \, : = \, \| u^\prime \|^2_{L^2(\R^+)} + \| u^\prime
\|^2_{L^2(\R^-)} - \beta^{-1}| 
u(0+)-u(0-) |^2   
\end{equation}
\vskip5pt
In both cases $\alpha=0$ and $\beta=0$ the $\delta$ and $\delta'$
respectively reduce to the free laplacian form. 

Besides, if $u$ belongs to the operator domain of a
$\delta^\prime$-interaction  with strength $- \beta$, then one has
\begin{equation} \label{deltaprime}
Q_{\beta}^{\delta'} (u) \, : = \, \| u^\prime \|^2_{L^2(\R^+)} + \| u^\prime
\|^2_{L^2(\R^-)} - \beta| u^\prime (0) |^2 
\end{equation}
which is the reason to attribute the name of $\delta'$ to
$H_{\beta}^{\delta'}$,  
that is, as recalled, an abuse of interpretation.
\vskip10pt
\p
For the Hamiltonian $H_{\tau}^{dip}$ the case 3. above applies
with $\omega=1$ and $a = \tau$. The energy
space is 
\begin{equation} \nonumber
D(Q_{\tau}^{dip}): = \{  u \in H^1 (\RE^-) \oplus H^1 (\RE^+), \ u (0+) = 
\tau u (0-) \}
\end{equation}
and the quadratic form is
\begin{equation} \nonumber
Q_{\tau}^{dip}  (u) \ = \  \| u^\prime \|^2_{L^2 (\RE^-)} + 
\| u^\prime \|^2_{L^2 (\RE^+)} . 
\end{equation}
\vskip10pt
\p

\subsection{Spectra} \label{spectra}
Here we recall the main spectral properties 
of the operators $H_{\alpha}^\delta$,  $H_{\beta}^{\delta'}$ 
and $H_{\tau}^{dip}$ (see \cite{[ABD]}). \vskip10pt\p
All of them have the essential spectrum which is purely
absolutely continuous and precisely  
$\sigma_{ess} (H_{\alpha}^\delta)=\sigma_{ess}
(H_{\beta}^{\delta'})=\sigma_{ac} (H_{\alpha}^\delta) 
=\sigma_{ac} (H_{\beta}^{\delta'})=\sigma_{ess}
(H_{\tau}^{dip})=\sigma_{ac} (H_{\tau}^{dip})
=[0,+\infty)$.\vskip10pt\p 
Concerning the discrete spectrum, if nonempty it is purely point, and
precisely one has \vskip10pt\p 
If $\alpha\leq 0$, then $\sigma_p(H_{\alpha}^\delta)=\emptyset$; \p
if $\alpha>0$, then there exists a unique eigenvalue, given by
$\sigma_p(H_{\alpha}^\delta)=\{-\frac{\alpha^2}{4}\}\ .$\p  
If $\beta\leq 0$, then
 $\sigma_p(H_{\beta}^{\delta'})=\emptyset$; \p
if $\beta>0$, then there exists a unique eigenvalue, 
given by $\sigma_p(H_{\beta}^{\delta'})=\{-\frac{4}{\beta^2}\}\ .$\p
For any $\tau \in \R$, $\sigma_p(H_{\tau}^{dip})=\emptyset\ $.
\vskip10pt\p 
For any $\alpha,\ \beta \in (0,\infty)$, the corresponding normalized eigenfunctions of 
$H_{\alpha}^\delta$ and $H_{\beta}^{\delta'}$
are given
by 
$$
\varphi_{\alpha}(x)=\left( \frac{\alpha}{2} \right)^{\frac{1}{2}}\
e^{-\frac{\alpha}{2}|x|}, 
\quad \quad \xi_{\beta}(x)=\left(\frac 2 {\beta} \right)^{\frac{1}{2}}
{\rm sign}(x) e^{-\frac 2 {\beta}|x|} 
$$
\vskip10pt
\p
In any case we consider, the singular continuous spectrum is empty: $\sigma_{sc} (H_{\alpha}^\delta)=
\sigma_{sc} (H_{\beta}^{\delta'})=\sigma_{sc} (H_{\tau}^{dip})=\emptyset$.
\vskip 10pt
\p
In view of application to the proof of Corollaries
  \ref{cor:delta},
\ref{cor:deltaprime}, and \ref{cor:Qtau},
we remark that the structure of the spectrum of the operators
$H_{\alpha}^\delta$, $H_{\beta}^{\delta'}$, and $H_{\tau}^{dip}$
immediately
shows that: \p
i) $Q_{\tau}$ is positive definite; \p
ii) 
\begin{equation} \label{defposdelta}
Q_{\alpha}^\delta (u) + \frac{{\alpha}^2}{4} 
\|u\|^2  \ \geq \ 0,
\end{equation}
and equality holds if and only if
$ u = \lambda u_\alpha$, for some $\lambda \in {\mathbb C}$.
\p
iii)
\begin{equation} \label{defposdeltaprime}
Q_{\beta}^{\delta'}  (u) + \frac{4}{{\beta}^2} \|u\|^2  \ \geq \ 0,
\end{equation}
and equality holds if and only if
$ u = \lambda v_\beta$, for some $\lambda \in {\mathbb C}$.


\begin{section}
{Appendix II: Construction of nonlinear stationary states for point
  interactions} 
In this appendix we review some useful results on existence and explicit construction of standing waves for the standard NLS on the halfline (Subsection
8.1), on NLS perturbed by a $\delta$ interaction (Subsection 8.2),
and by a $\delta'$-interaction (Subsection 8.3). They are mostly known, but we prefer to give a 
selfconsistent treatment. 
Finally, we give new results for the NLS with a
dipole interaction (Subsection 8.4). Main references are
\cite{[BL],[C],[CL]} for the standard case, \cite{[FOO], [FJ]} for the
delta-like perturbation, and \cite{[AN2]} for the $\delta'$ potential.
In particular, for a complete proof of the 
identification of the ground states in the
latter case we refer to \cite{[AN2]}. 

\n
We warn the reader that along this Appendix we shall always consider {\em
  real} solutions to the stationary Schr\"odinger equation only. As
the equation \eqref{pointNLS} is genuinely complex, of course other stationary
states exist and are found by exploiting phase invariance.

\subsection{The Cauchy problem for the stationary NLS on the halfline}

\n
In the present section we give, for completeness, the proof that every
standing wave of a NLS on the line with a point interaction is
constructed by matching two truncated standing waves on the line with
suitably chosen parameters (centre, amplitude and phase). This is the
way standing waves of NLS with $\delta$, $\delta'$ and dipole
interactions are obtained. Here we prove that the procedure is general
and we show how to apply it to the determination of standing waves of
NLS with virtually every point interactions.\p

We start giving some elementary properties of the solution to
  the equation
\be \label{nlsstand}
-u''+\omega u=u|u|^{p-1}, \hbox{ } \omega>0.
\ee

\begin{lem}\label{unodim} Let $u$ any solution to \eqref{nlsstand}. Then the following properties hold:\par\noindent
a) $u$ satisfies a conservation law:
\be \label{energmec}
(u'(x))^2 -\omega (u(x))^2 + \frac 2{p+1}|u(x)|^{p+1}= {\rm{const}},
\qquad \forall x \ {\mbox{in the domain of }} u;
\ee
b) if $u$ is a maximal solution to equation \eqref{nlsstand}, then it is defined on $\R$; \par\noindent
c) if $u$ is a solution to \eqref{nlsstand} defined in the
interval
$[a, + \infty)$ such that $\lim_{x \to \infty} u(x) = 0$, then
it must satisfy
\be \label{energyzero}
(u'(x))^2 -\omega (u(x))^2 + \frac 2{p+1}|u(x)|^{p+1}= 0,
\qquad \forall x \in [a, + \infty).
\ee
\end{lem}

\begin{proof} 
Indeed, for any $ x $ in the domain of $u$,
\be \nonumber \begin{split}
  \f d {dx} [(u'(x))^2 - \omega (u(x))^2 + \frac 2{p+1}|u(x)|^{p+1}]
\ = \ &  2 \, u'(x) [u''(x)^2 - \omega u(x) + |u(x)|^{p-1}u(x)],
\end{split} \ee
that vanishes since $u$ is a solution to \eqref{nlsstand}. This proves a). 
Moreover from \eqref{energmec} one immediately has that any
maximal solution $u$ has to be
bounded,
otherwise $(u')^2$ would become negative at some $x$. Furthermore,
again
from \eqref{energmec}, $u'$ has to be bounded too. Then, if the domain  
of $u$ is bounded, then it can be continued, contradicting the maximality
of $u$. As regards c), by \eqref{energmec}
$u'(x)$ tends to a constant as $x$ goes to infinity, but in order to guarantee
$u(x) \to 0$, such a constant must be equal to zero, and the proof is
complete.  
\end{proof}
\begin{remark}
Any solution to the Cauchy problem
\be \nonumber 
\left \{ \begin{array}{c}
-u''+\omega u=u|u|^{p-1}, \hbox{ } \omega>0\\
u(0)=a, u'(0)=b \end{array} \right. \ee
satisfies
\begin{equation} \nonumber 
\omega a^2 - \frac 2{p+1}|a|^{p+1}=b^2\ .
\end{equation}
\end{remark}
\vskip5pt
\par\noindent

We introduce for shorthand the following notation
\be \label{solitone}
\varphi_{p,\omega}(x ) \ =  \ \left( \frac{\omega (p+1)}{2\cosh^2
  (\frac{p-1}{2} \sqrt \omega x)}\right)^\frac 1{p-1}.
\ee
\begin{thm}\label{rappresentazione}
Given $a \in \R$, let $u$ be a non-trivial
solution to \eqref{nlsstand} on the interval $[a, + \infty )$
such that $\lim_{x \to \infty} u(x) = 0$.
Then,
\begin{equation}\nonumber 
| u (x)|  \ = \  \varphi_{p,\omega}(x + x_0)
\end{equation}
for some $x_0\in \R$.
\end{thm}
\begin{proof}
By Lemma \ref{unodim}, $u$ must satisfy the condition
\eqref{energyzero}, so $|u (a) | \leq ( \om (p+1)/2 )^{\f 1 {p-1}}$ and 
$$
| u'(a)| \ = \  \sqrt{\omega (u(a))^2 + \frac 2{p+1}|u(a)|^{p+1}}.
$$
First, notice that $\varphi_{p, \om} (\R) = (0,   ( \om (p+1)/2
)^{\f 1 {p-1}}]$,
so there exists $y_0 \in \R^+$ s.t. $\varphi_{p, \om} (y_0) \ =
\ |u (a) |$. Besides, one can directly check  that 
$$
(\varphi_{p, \om})^\prime (y_0) \ = \ - | u' (a) |.
$$
Now, observing that $\varphi_{p, \om}$ is even, and that for any
$y \in \R$ the
functions  $\pm \varphi_{p, \om} (\cdot + y)$
solve equation  \eqref{nlsstand}, we conclude that:

\n
-- If $u(a) > 0$ and $u'(a) \geq 0$, then $u (x) = \varphi_{p, \om}
(x - y_0 - a)$.

\n
-- If $u(a) > 0$ and $u'(a) \leq 0$, then $u (x) = \varphi_{p, \om}
(x + y_0 - a)$.

\n
-- If $u(a) < 0$ and $u'(a) \geq 0$, then $u (x) = - \varphi_{p, \om}
(x + y_0 - a)$.

\n
\n
-- If $u(a) < 0$ and $u'(a) \leq 0$, then $u (x) = - \varphi_{p, \om}
(x - y_0 - a)$.

\n
The theorem is proven.
\end{proof}

In the next subsections, we follow the previous analysis of the
Cauchy problem for NLS on the halfline, and construct the families of
stationary states for the three examples of point interactions we are studying.\p
\vskip 5pt

\subsection{Stationary states for the $\delta$ potential}

\n
 Here we explicitly give the solutions to the following
ODE
\begin{equation}\label{systemdelta}\left \{
\begin{array}{c} -u''+\omega u= u|u|^{p-1} \hbox{ on } \R\setminus
  \{0\}, u(x) \in \R , \omega \in \R\\ 
u\in H^1(\R)\cap \Big(H^2(\R^-)\oplus H^2(\R^+) \Big)\\u'(0+)- u'(0-)=
- \alpha u(0)
\end{array}\right. \end{equation}
First we prove a non-existence result.
\begin{prop}\label{nonexistencedelta}
Assume $\omega\leq \alpha^2/4$, then  the only solution to
  \eqref{systemdelta} is the trivial one. 
\end{prop}
{\bf Proof.}
First notice that by Theorem \ref{rappresentazione}
a solution $u$ to \eqref{systemdelta} 
is described by two pieces of solitons matched at the origin, and by
the continuity condition (recall 
that we are assuming $u\in H^1(\R)$) they have constant sign. For simplicity we assume $u(x)>0$
for every $x\in \R$.
After multiplication of \eqref{systemdelta} by 
$\varphi_\alpha(x)=\sqrt{\frac \alpha 2}e^{-\frac{\alpha}2 |x|}$, 
that is a normalized eigenvector of the attractive $\delta$
interaction, already defined in the  proof of Corollary \ref{cor:delta}, 
  we
integrate twice  by parts and get the identity
$$-\frac{\alpha^2}4 \int_\R \varphi_\alpha u dx+\omega
\int_\R\varphi_\alpha u dx=\int_{\R}\varphi_\alpha u |u|^{p-1} dx >0$$ 
(where we have used the fact that $u(x)$ has a constant sign) and hence necessarily 
$\omega>\frac{\alpha^2}{4}$.

\hfill$\Box$ 

\begin{figure}
\begin{center}
\subfloat[][\emph{$\omega=1 $}]
{\includegraphics[width=.40\columnwidth]{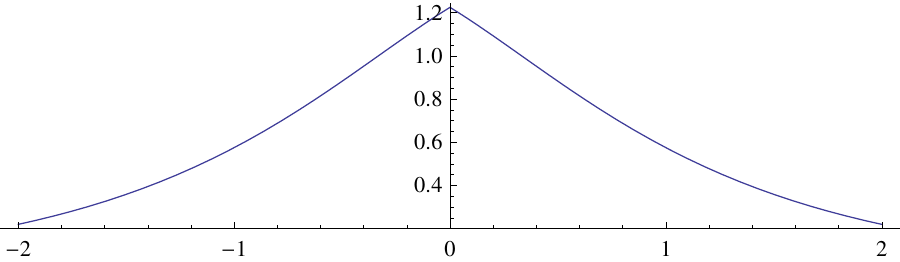}} \quad
\subfloat[][\emph{$\omega=9$}]
{\includegraphics[width=.40\columnwidth]{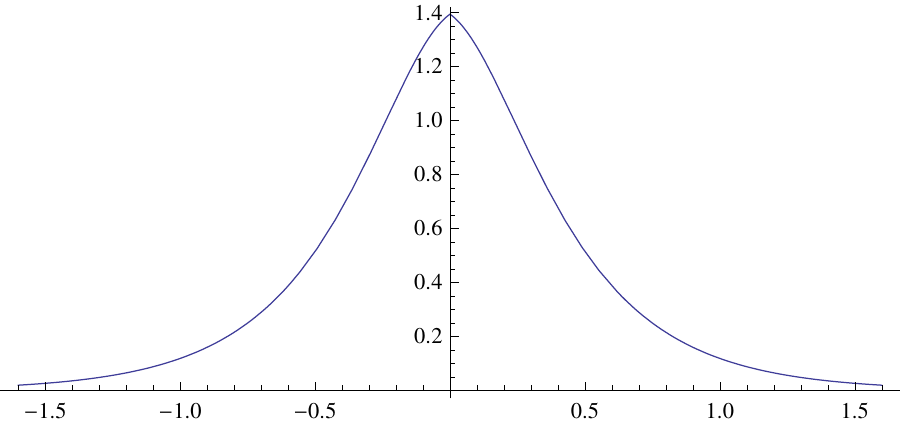}} \\
\caption{p=3 NLS with $\delta$ potential. $\alpha=1$; $\omega=1$ and $\omega=9$}
\label{fig:subfig}
\end{center}
\end{figure}

\begin{prop}\label{delta}
For every $\alpha>0$ and for every $\omega>\alpha^2/4$ there exist exactly two solutions
to \eqref{systemdelta}, given by $\pm u_{\alpha,p, \omega}$, where the
function $u_{\alpha,p, \omega}$ was defined in \eqref{u+delta}, \eqref{xtilde}.
\end{prop}

{\bf Proof.}
According to Theorem
\ref{rappresentazione} and by the continuity condition on $u$ (indeed we assume $u(x) \in H^1(\R)$)
we deduce
that either $u(x)>0$ or $u(x)<0$ for every $x\in \R$.
We assume that $u(x)>0$ (the case $u(x)<0$ is similar).


Again by Theorem \ref{rappresentazione} there exist $x_\pm\in \R$ such that
$$u(x)= \varphi_{p, \om}  (x+x_\pm ), \qquad x\in \R^\pm,$$ 
where $\varphi_{p, \om}$ was defined in \eqref{solitone}.


\noindent
By imposing the continuity condition at the the origin we deduce (due
to the shape of the function $\varphi_{p, \om}$),
that 
$x_+=\pm x_-$. In the case $x_+=x_-=\tilde x$ we get
$u(x)= \varphi_{p, \om}(x + \tilde x)$,
that can be excluded since the derivative at the origin has no
jump, so, as $\alpha > 0$, the boundary condition in
\eqref{systemdelta} is not satisfied. 
Hence we have
$x_+=-x_-$.
By the boundary condition imposed by \eqref{systemdelta} on the
derivative of $u$, we deduce that,
denoting $\tilde x=-x_+$, 
$$\varphi_{p,\omega}'(\tilde x) - \varphi_{p,\omega}'(-\tilde x)= - 
\alpha \varphi_{p,\omega}(\tilde x),$$
namely
$$-2 \sqrt \omega \Big (\frac{1}{\cosh^2 (\frac{p-1}{2} \sqrt \omega
  \tilde x)}\Big)^\frac 1{p-1} \tanh  
 (\frac{p-1}{2} \sqrt \omega \tilde x)= - \alpha \Big (\frac{1}
 {\cosh^2 (\frac{p-1}{2}  \sqrt \omega \tilde x )}
 \Big)^\frac 1{p-1}$$ 
(where we used the even character of the function $\varphi_\omega$), i.e.
$\tanh (\frac{p-1}{2} \sqrt \omega \tilde x)=\frac{\alpha}{2\sqrt \omega}$.
The proof is complete.

\hfill$\Box$

\par\noindent
The stationary states for a $\delta$ interaction are represented in Figure 1.\par\noindent
We immediately have the following result (see \cite{[FOO]})
\begin{cor}\label{monotonL2delta}
The function $$(\alpha^2/4, \infty) \ni \omega \rightarrow
\|u_{\alpha, p, \omega}\|_{L^2(\R)} 
\in (0, \infty)$$
is continuous, increasing and surjective.
\end{cor}
\begin{proof}
Using \eqref{xtilde}, by direct computation one gets
$$
\|u_{\alpha, p, \omega}\|_{L^2(\R)}^2 \ = \ C(p)
\omega^{\f{5-p}{2p-2}} \int_{\f \alpha {2 \sqrt \om}}^1 (1-t^2)^{\f{3-p}{p-1}}, 
$$
where $C(p) > 0$ is independent of $\omega$,
that is obviously a monotonically increasing function of $\omega$,
approaching zero as $\om$ vanishes, and going to infinity as $\omega$
goes to infinity. 
\end{proof}

\subsection{Stationary states for the $\delta'$ interaction}
\vskip 5pt

We study the problem
\begin{equation}\label{delta'ODE}
\left \{\begin{array}{c} -u''+\omega u=u|u|^{p-1}  \hbox{ on } 
    \R\setminus \{0\}, \ u(x) \in \R, \ \omega \in \R\\
u\in H^2(\R^+)\oplus H^2(\R^-)\\
u(0-)-u(0+)=\beta u'(0), u'(0+)=u'(0-)\\ u(0+)\cdot u(0-)<0
\end{array} \right.
\end{equation}

\noindent
First, we prove a nonexistence result.
\begin{prop}\label{nonexistencedeltaprime}
If $\omega\leq \beta^2/4$, then 
 the problem
to \eqref{delta'ODE} admits the trivial solution $u \equiv 0$ only.
\end{prop}
{\bf Proof.}
First notice that by Theorem \ref{rappresentazione}
any solution $u$ to \eqref{delta'ODE} consists of two
pieces of solitons suitably matched  at the origin.
Moreover, by the boundary condition they have opposite sign on the
real half-lines $\R^\pm$, so
we can assume $u(x)>0$ on $\R^+$, being the case
$u(x)<0$ equivalent.
After multiplication of \eqref{delta'ODE} by  the function
$\varphi_\beta(x)= \sqrt{\frac 2\beta} ({\mbox{sign}} x)e^{-\frac
  2\beta |x|}$ (that is a normalized eigenvector of the attractive
$\delta'$ interaction and was 
defined in the proof of Corollary \ref{cor:deltaprime}),
and integrating by parts twice, we get
$$-\frac4{\beta^2}\int v\cdot \varphi_\beta dx +\omega \int
\varphi_\beta \cdot v \, dx \ = \ \int \varphi_\beta \cdot 
v|v|^{p-1} dx>0$$
where we have used the fact that $v \varphi_\beta$ has constant sign,
 and hence necessarily 
$\omega>\frac{\beta^2}{4}$.

\hfill$\Box$ 

\begin{prop}\label{deltaprimeodd}
Let $\beta>0$. 
Then, under the extra assumption 
\begin{equation}\label{extraasym}
u(0+)=-u(0-)>0,
\end{equation}
there exists 
a solution
to \eqref{delta'ODE}
if and only if $\omega>\frac 4{\beta^2}$.
Moreover this solution is unique and 
equals $u_{odd, \beta, p, \omega}={\rm sign}(x)\varphi_{p, \om} (x+\bar x)$,
where ${\rm sign}(x)$ is the sign function, $\varphi_{p, \om}$ was defined in 
\eqref{solitone}, and $\bar x=\bar x(\beta, p, \omega)>0$ solves the equation $$\beta \sqrt \omega \tanh 
 \left( \frac{p-1}{2} \sqrt \omega \bar x \right)=2,\ \bar x > 0\ .$$

\end{prop}

{\bf Proof.} By Theorem \ref{rappresentazione}
any solution $u$ that satisfies \eqref{delta'ODE} 
plus the extra assumption $u(0+)>0$ has necessarily the following structure
$$ \pm \varphi_{p, \om} (x + \bar x), \qquad  x\in \R^\pm,$$
%
where $x_\pm$ are to be chosen in order to satisfy the boundary
conditions in \eqref{delta'ODE}.
Due to \eqref{extraasym} and to the continuity of the derivative, we conclude
$x_+=-x_-$.
By introducing $x_+=\bar x$, the condition on the jump of $u$ at zero
(see \eqref{deltaprime}) prescribes
$$\varphi_{p,\omega}(\bar x)+\varphi_{p,\omega}(-\bar x) =-\beta \varphi_{p,\omega}'(\bar x),$$
or, more explicitly,
$$
\frac{1}{\cosh^\frac 2{p-1} (\frac{p-1}{2} \sqrt \omega \bar x )}
+\frac{1}{\cosh^\frac 2{p-1} (\frac{p-1}{2} \sqrt \omega (-\bar x) )}
=\beta \sqrt \omega \frac{ \tanh 
 (\frac{p-1}{2} \sqrt \omega \bar x)}{\cosh^\frac 2{p-1}
 (\frac{p-1}{2} \sqrt \omega \bar x)} ,
$$
which implies $\bar x>0$. 

\n
The proof is complete.
\hfill$\Box$

\begin{prop}\label{deltaprimeas}
Let $\beta>0$. Then there exists
a solution to \eqref{delta'ODE} under the extra assumption 
\begin{equation}\label{iextra} u(0+)>0
\hbox{ and } u(0+)>|u(0-)|\end{equation}
if and only if $\omega>\frac{4}{\beta^2}\frac{p+1}{p-1}$. 
Moreover the solution 
to \eqref{delta'ODE} that satisfies the extra assumptions
\eqref{iextra} is unique and equals $u_{as, \beta, p, \omega}=\pm\varphi_{p, \om} (x+ x_\pm), x\in \R^{\pm}$,
where the function $\varphi_{p, \om}$ was defined in
\eqref{solitone} and $x_\pm$ are the only solutions to the system
\begin{equation} \label{tsystem}
\left\{ \begin{array}{c}
\frac{ \tanh 
(\frac{p-1}{2} \sqrt \omega x_+)}{\cosh^\frac 2{p-1} (\frac{p-1}{2} \sqrt \omega x_+)}+
\frac{\tanh 
(\frac{p-1}{2} \sqrt \omega x_-)}{\cosh^\frac 2{p-1} (\frac{p-1}{2} \sqrt \omega x_-)}=0
\\
\frac{1}{\cosh^\frac 2{p-1} (\frac{p-1}{2} \sqrt \omega x_+ )}
+\frac{1}{\cosh^\frac 2{p-1} (\frac{p-1}{2} \sqrt \omega x_- )}
=\beta \sqrt \omega \frac{ \tanh 
(\frac{p-1}{2} \sqrt \omega x_+)}{\cosh^\frac 2{p-1} (\frac{p-1}{2} \sqrt \omega x_+)}
\end{array} \right.
\end{equation}

 that satisfy the condition $ x_-<0<x_+<|x_-|$. 
\end{prop}

{\bf Proof.}
By Theorem \ref{rappresentazione}
any solution $u$ that satisfies \eqref{delta'ODE} 
plus the extra assumption $u(0+)>0$ is necessarily of the type
$ \pm \varphi_{p,\om} (\cdot + x_\pm), x \in \R^\pm$, 
where $x_\pm$ are to be chosen in order to satisfy the boundary conditions.
It is also easy to check that under our assumptions
necessarily $u'(0)<0$ and hence
$x_+>0$ and $x_-<0$. Moreover, since we are assuming 
$|u(0+)|>|u(0-)|$, then $x_+<|x_-|$.
In fact, the boundary conditions are equivalent to 
$$\varphi_{p,\omega}'(x_+)+\varphi_{p,\omega}'(x_-)=0$$
$$\varphi_{p,\omega}(x_+)+\varphi_{p,\omega}(x_-) =-\beta \varphi_{p,\omega}'(x_+),$$
and system above rephrases as 
$$\frac{ \tanh 
 (\frac{p-1}{2} \sqrt \omega x_+)}{\cosh^\frac 2{p-1} (\frac{p-1}{2} \sqrt \omega x_+)}+
\frac{\tanh 
 (\frac{p-1}{2} \sqrt \omega x_-)}{\cosh^\frac 2{p-1} (\frac{p-1}{2} \sqrt \omega x_-)}=0
$$
$$
\frac{1}{\cosh^\frac 2{p-1} (\frac{p-1}{2} \sqrt \omega x_+ )}
+\frac{1}{\cosh^\frac 2{p-1} (\frac{p-1}{2} \sqrt \omega x_- )}
=\beta \sqrt \omega \frac{ \tanh 
 (\frac{p-1}{2} \sqrt \omega x_+)}{\cosh^\frac 2{p-1} (\frac{p-1}{2} \sqrt \omega x_+)}
$$
with the extra conditions $$x_-<0<x_+<|x_-|.$$
According to Proposition 5.1, Lemma 5.2, and Theorem 5.3 in
  \cite{[AN2]}, to which we refer for details,  the above system has a unique solution $(x_+, x_-)$.
\hfill$\Box$

\par\noindent
The situation is depicted in Figure 2, where the odd and asymmetric
stationary states for a cubic NLS plus a $\delta'$ interaction with
$\beta=1$ and $\omega=64$ are represented. 
\par\noindent
Next, we collect some properties of the elements
$u_{odd, \beta, p, \omega}$ and $u_{as, \beta, p, \omega}$ of the two families of standing waves of NLS with $\delta'$ interaction.

\begin{prop}\label{monotonL2}
The following properties hold:
\begin{itemize}
\item the function $$(4/\beta^2, \infty) \ni \omega\rightarrow
  \|u_{odd, \beta, p, \omega}\|_{L^2(\R)} 
\in (0, \infty)$$
is continuous, increasing and surjective;
\item  the function $$\left[ \frac{4(p+1)}{\beta^2(p-1)}, \infty \right) \ni \omega
\rightarrow \|u_{as, \beta, p, \omega}\|_{L^2(\R)}\in 
[\|u_{odd, \beta, p, \frac{4(p+1)}{\beta^2(p-1)}}\|_{L^2(\R)}, \infty)$$
is continuous, increasing and surjective.
\end{itemize}
\end{prop}

{\bf Proof.}
The result immediately follows 
from Proposition 6.5 in \cite{[AN2]}. 

\hfill$\Box$

Next result is useful to compare energy and mass of stationary states $u_{odd, \beta, p, \omega}$ and $u_{as, \beta, p, \omega}$.
\begin{prop}\label{L2deltaprim'}
We have the following identities
\begin{equation} \nonumber 
\|u_{odd, \beta, p, \omega}\|_{L^2(\R)}^2=\Big (\frac {\omega
  (p+1)}2\Big )^\frac 2{p-1} \frac{2 }{(p-1) \sqrt \omega}\times 
\end{equation}
$$\left (\int^1_{-1} (1-t^2)^\frac{3-p}{p-1} dt 
- \int_{-\frac 2 {\beta \sqrt \omega}}^{\frac 2 {\beta \sqrt \omega}}
 (1-t^2)^\frac{3-p}{p-1} dt\right ) $$
and 
\begin{equation} \nonumber 
\|u_{as, \beta, p, \omega}\|_{L^2(\R)}^2=\Big (\frac {\omega (p+1)}2\Big )^\frac 2{p-1} \frac{2 }{(p-1) \sqrt \omega}\times
\end{equation}
$$\left (\int^1_{-1} (1-t^2)^\frac{3-p}{p-1} dt 
- \int_{\tanh 
 (\frac{p-1}{2} \sqrt \omega x_-)}^{\tanh 
 (\frac{p-1}{2} \sqrt \omega x_+)} (1-t^2)^\frac{3-p}{p-1} dt\right )
$$
Moreover, we have
\begin{equation}\nonumber 
\|u_{odd, \beta, p, \omega}\|_{L^{p+1}(\R)}^{p+1}=\end{equation}
$$\Big (\frac {\omega (p+1)}2\Big )^\frac {(p+1)}{p-1} \frac{2 }{(p-1) \sqrt \omega}
\left (\int^1_{-1} (1-t^2)^\frac{2}{p-1} dt 
- 
 \int_{-\frac 2 {\beta \sqrt \omega}}^{\frac 2 {\beta \sqrt \omega}}
(1-t^2)^\frac{2}{p-1} dt\right ) 
$$
and 
\begin{equation} \nonumber 
\|u_{as, \beta, p, \omega}\|_{L^{p+1}(\R)}^{p+1}=\end{equation}
$$\Big (\frac {\omega (p+1)}2\Big )^\frac {(p+1)}{p-1} \frac{2 }{(p-1) \sqrt \omega}
\left (\int^1_{-1} (1-t^2)^\frac{2}{p-1} dt 
-
 \int_{-\tanh 
 (\frac{p-1}{2} \sqrt \omega x_-)}^{\tanh 
 (\frac{p-1}{2} \sqrt \omega x_+)} (1-t^2)^\frac{2}{p-1} dt\right )
$$
\end{prop}
{\bf Proof.} By looking at the expression of $u_{odd, \beta, p, \omega}(x)$
we get:
$$\|u_{odd, \beta, p, \omega}(x)\|_{L^2(\R)}^2=$$
$$ \Big (\frac {\omega (p+1)}2\Big )^\frac 2{p-1}
\Big (\int_{\bar x}^\infty 
\frac 1{\cosh^\frac 4{p-1} (\frac{p-1}{2} \sqrt \omega x)} dx + 
\int^{-\bar x }_{-\infty} 
\frac 1{\cosh^\frac 4{p-1} (\frac{p-1}{2} \sqrt \omega x)} dx \Big )$$
$$= 
\Big (\frac {\omega (p+1)}2\Big )^\frac 2{p-1} \frac{2 }{(p-1) \sqrt \omega}
\Big (\int_{\frac{p-1}{2} \sqrt \omega \bar x}^\infty 
\frac 1{\cosh^\frac 4{p-1} (x)} dx + 
\int^{-\frac{p-1}{2} \sqrt \omega \bar x}_{-\infty} 
\frac 1{\cosh^\frac 4{p-1} (x)} dx \Big )$$
and after the change of variable $t=\tanh x$ we get
$$\|u_{odd, \beta, p, \omega}(x)\|_{L^2(\R)}^2=\Big (\frac {\omega (p+1)}2\Big )^\frac 2{p-1} \frac{2 }{(p-1) \sqrt \omega}\times$$$$
\left (\int^1_{\tanh 
 (\frac{p-1}{2} \sqrt \omega \bar x)} (1-t^2)^\frac{3-p}{p-1} dt 
+ \int_{-1}^{-\tanh 
 (\frac{p-1}{2} \sqrt \omega \bar x)} (1-t^2)^\frac{3-p}{p-1} dt\right )
$$
The other identities can be proved in the same way.

\hfill$\Box$

\begin{figure}
\begin{center}
\subfloat[][\emph{$\omega=64$, odd}]
{\includegraphics[width=.40\columnwidth]{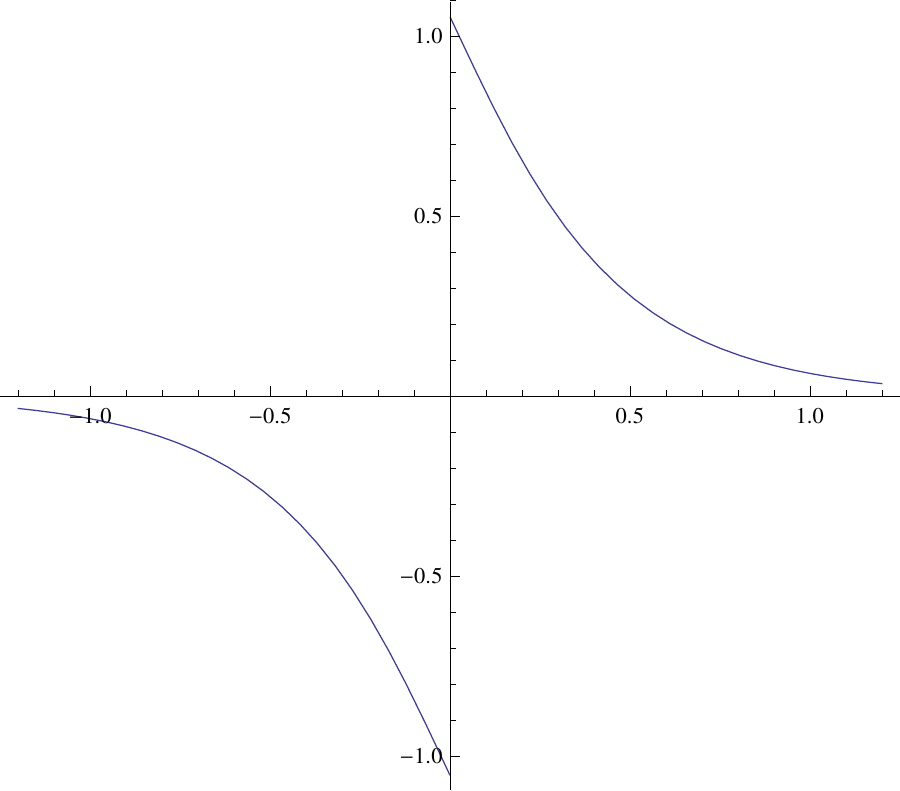}} \quad
\subfloat[][\emph{$\omega=64$, asymmetric}]
{\includegraphics[width=.40\columnwidth]{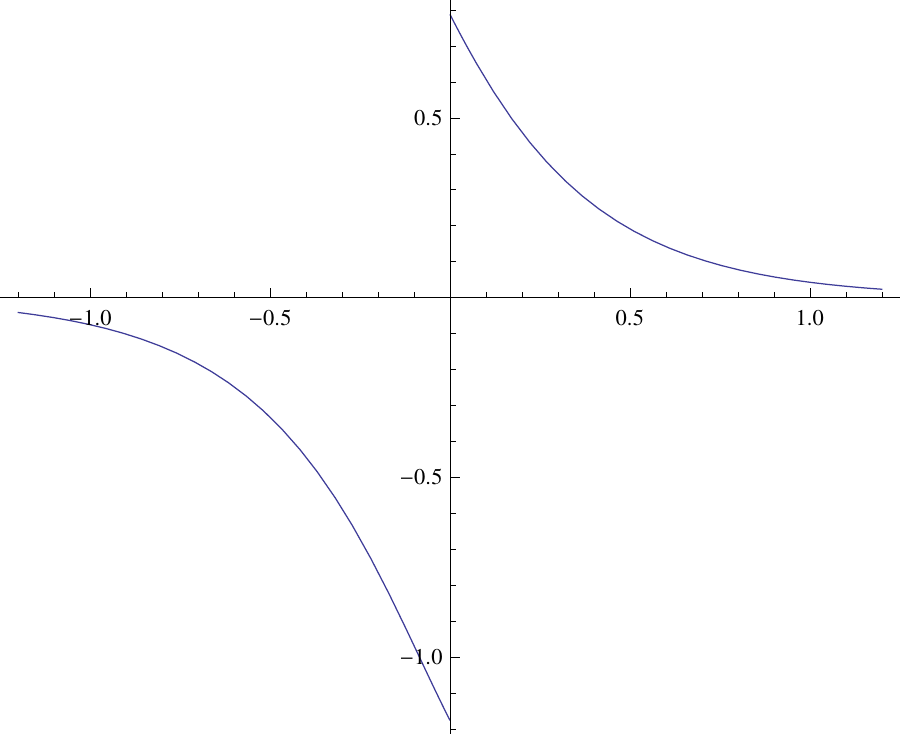}} \\
\caption{p=3 NLS with $\delta'$: $u_{odd}$ and $u_{as}$ for $\beta=1$, $\omega=64.$}
\label{fig:subfig}
\end{center}
\end{figure}

\subsection{Stationary states for the dipole interaction}

We study
\begin{equation}\label{QtauODE}
\left \{\begin{array}{c} -u''+\omega u=u|u|^{p-1} \hbox{ on }
    \R\setminus \{0\}, u(x) \in \R, \omega>0\\ 
u\in H^2(\R^-)\oplus H^2(\R^+)\\
u(0+)=\tau u(0-), \tau u'(0+)=u'(0-)
\end{array} \right.
\end{equation}
\begin{prop}\label{Qtau}
For every $\tau>0$ and $\omega>0$
there exist exactly two solutions to \eqref{QtauODE} under the extra assumption
\begin{equation} \nonumber 
u(0+)>0
\end{equation}
Moreover, the  solutions 
have the following structure
$$\chi_{\tau, p, \omega}^+(x)=
\left(\frac{\omega (p+1)}{2\cosh^2 (\frac{p-1}{2} \sqrt \omega (x +
  \xi_\pm))}\right)^\frac 1{p-1} 
\hbox{ for } x \in \R^\pm$$
$$\chi_{\tau, p, \omega}^-(x)=
\left (\frac{\omega (p+1)}{2\cosh^2 (\frac{p-1}{2} \sqrt \omega (x -
  \xi_\pm))}\right)^\frac 1{p-1} 
\hbox{ for } x \in \R^\pm$$
where $\xi_\pm=\xi_\pm(\tau, p, \omega)\in \R$ are defined by
the following conditions: 
$$\tanh \left(\frac{p-1}{2} \sqrt \omega \xi_+ \right)=\sqrt
\frac{1-\tau^{p-1}}{1-\tau^{p+3}}$$ 
$$\tanh \left( \frac{p-1}{2} \sqrt \omega \xi_- \right)=\tau^2\sqrt
\frac{1-\tau^{p-1}} {1-\tau^{p+3}}$$
\end{prop}

{\bf Proof.}
By Theorem \ref{rappresentazione}
any solution $u(x)$ that satisfies \eqref{QtauODE} 
plus the extra assumption $u(0+)>0$ has necessarily the following structure
$$u(x)=\left(\frac{\omega (p+1)}{2\cosh^2 (\frac{p-1}{2} 
\sqrt \omega (x+y_\pm )}\right)^\frac 1{p-1} \hbox{ for } x\in \R^\pm$$
where $y_\pm$ have to be chosen in order to satisfy the boundary conditions.
$$\tau \varphi_{p,\omega}'(y_+)-\varphi_{p,\omega}'(y_-)=0$$
$$\varphi_{p,\omega}(y_+) - \tau \varphi_{p,\omega}(y_-) =0$$
where $\varphi_{p,\omega}$ is given in \eqref{solitone}.
The above system is equivalent to:
$$\tau \frac{ \tanh 
 (\frac{p-1}{2} \sqrt \omega y_+)}{\cosh^\frac 2{p-1} (\frac{p-1}{2} \sqrt \omega y_+)}-
\frac{\tanh 
 (\frac{p-1}{2} \sqrt \omega y_-)}{\cosh^\frac 2{p-1} (\frac{p-1}{2} \sqrt \omega y_-)}=0
$$
$$
\frac{1}{\cosh^\frac 2{p-1} (\frac{p-1}{2} \sqrt \omega y_+ )}
-\tau \frac{1}{\cosh^\frac 2{p-1} (\frac{p-1}{2} \sqrt \omega y_- )}
=0
$$
By introducing $t_\pm=\tanh (\frac{p-1}{2} \sqrt \omega y_\pm)$  and by the well-known identity
$\frac 1{\cosh^2 x}= 1 - \tanh^2 x$ the system above is equivalent to
\be\left \{ \begin{array}{c}\tau^2 t_+=t_-\\
(1-t_+^2)=\tau^{p-1}(1-t_-^2)
\end{array} \right. \ee
and hence the conclusion easily follows.

\hfill$\Box$

The ground states of a dipole interaction with various values of $\tau$ and $\omega=1$ are represented in Figure 3.

\begin{figure}
\begin{center}
\subfloat[][\emph{$\tau=2,\omega=1\ .$}]
{\includegraphics[width=.40\columnwidth]{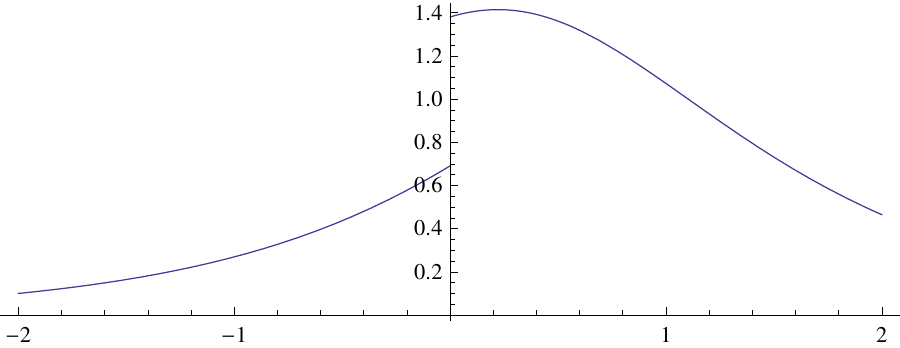}} \quad
\subfloat[][\emph{$\tau=\frac{1}{2},\omega=1$}]
{\includegraphics[width=.40\columnwidth]{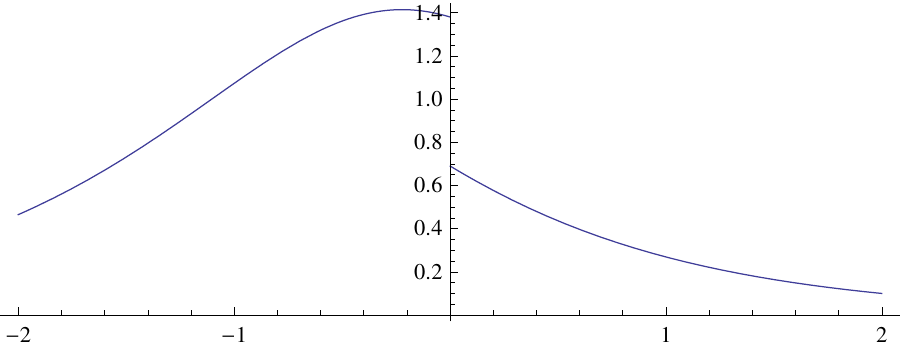}} \\
\subfloat[][\emph{$\tau=-2,\omega=1$}]
{\includegraphics[width=.40\columnwidth]{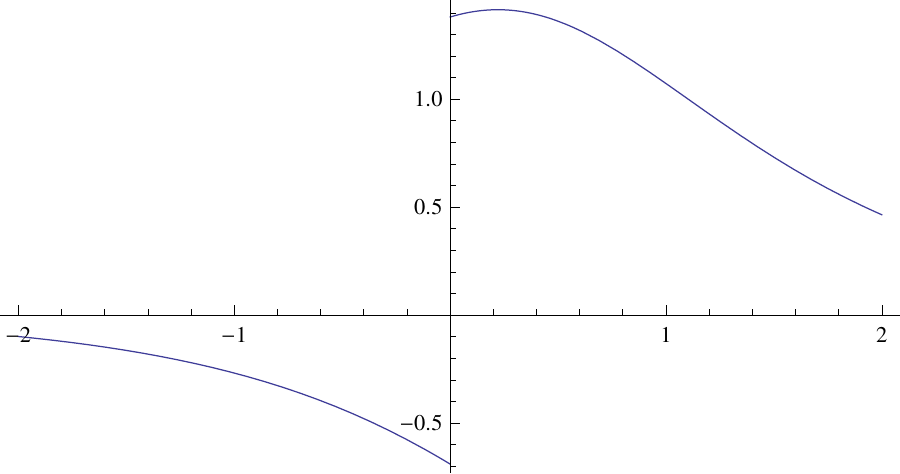}} \quad
\subfloat[][\emph{$\tau=-\frac{1}{2}, \omega=1$}]
{\includegraphics[width=.40\columnwidth]{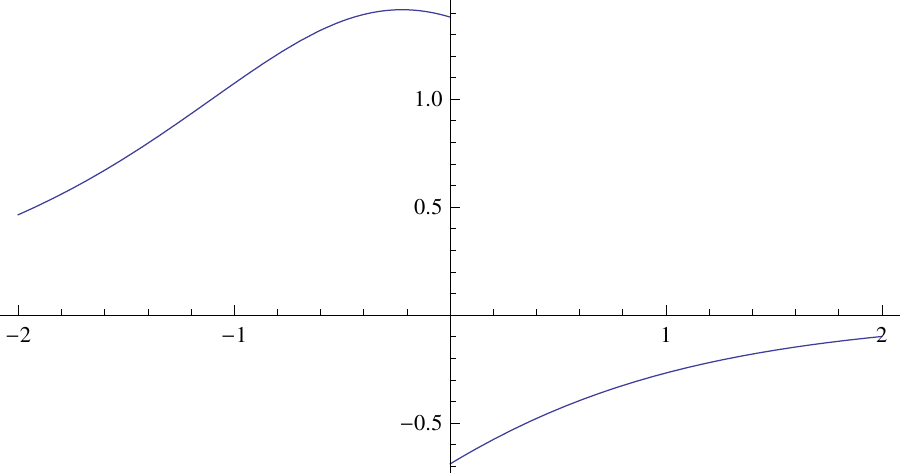}}
\caption{Ground states for the $p=3$ NLS with dipole.}
\label{fig:subfig}
\end{center}
\end{figure}

The following property allows to compare energy and mass of stationary states the of $H_{tau}^{dip}$.

\begin{prop}\label{L2Qtau}
Assume $\tau>1$ and $\chi^\pm_{\tau, p, \omega}$ as in Proposition \ref{Qtau}, then we have the following identities:
\begin{equation}\label{L2}\|\chi^\pm_{\tau, p, \omega}\|_{L^2(\R)}^2=\end{equation}
$$\Big (\frac {\omega (p+1)}2\Big )^\frac 2{p-1} \frac{2 }{(p-1) \sqrt \omega}
\left ( \int_{-1}^1 (1-t^2)^\frac{3-p}{p-1}  \pm  \int^{\tau^2 \sqrt \frac{1-\tau^{p-1}}{1-\tau^{p+3}}}_{\sqrt \frac{1-\tau^{p-1}}{1-\tau^{p+3}}} (1-t^2)^\frac{3-p}{p-1} dt \right )
$$
\begin{equation}\label{Lp+1}\|\chi^\pm_{\tau, p, \omega}\|_{L^{p+1}(\R)}^{p+1}=\end{equation}
$$\Big (\frac {\omega (p+1)}2\Big )^\frac {(p+1)}{p-1} \frac{2 }{(p-1) \sqrt \omega}
\left (\int^1_{-1} (1-t^2)^\frac{2}{p-1} dt 
\pm 
 \int_{\sqrt \frac{1-\tau^{p-1}}{1-\tau^{p+3}}}^{\tau^2 \sqrt
   \frac{1-\tau^{p-1}}{1-\tau^{p+3}}} (1-t^2)^\frac{2}{p-1}dt\right )  
$$
\end{prop}
{\bf Proof.}
We prove \eqref{L2}. By looking at the explicit expression of $\chi^+_{\tau, p, \omega}$ we get
$$\|\chi^+_{\tau, p, \omega}\|_{L^2(\R)}^2=$$
$$ \Big (\frac {\omega (p+1)}2\Big )^\frac 2{p-1}
\Big (\int_{\xi_+}^\infty 
\frac 1{\cosh^\frac 4{p-1} (\frac{p-1}{2} \sqrt \omega x)} dx + 
\int^{\xi_-}_{-\infty} 
\frac 1{\cosh^\frac 4{p-1} (\frac{p-1}{2} \sqrt \omega x)} dx \Big )$$
$$= 
\Big (\frac {\omega (p+1)}2\Big )^\frac 2{p-1} \frac{2 }{(p-1) \sqrt \omega}
\Big (\int_{\frac{p-1}{2} \sqrt \omega \xi_+}^\infty 
\frac 1{\cosh^\frac 4{p-1} (x)} dx + 
\int^{\frac{p-1}{2} \sqrt \omega \xi_-}_{-\infty} 
\frac 1{\cosh^\frac 4{p-1} (x)} dx \Big )$$
and after the change of variable $t=\tanh x$ we get
$$\|\chi^+_{\tau, p, \omega}\|_{L^2(\R)}^2=\Big (\frac {\omega
  (p+1)}2\Big )^\frac 2{p-1} \frac{2 }{(p-1) \sqrt \omega}\times$$
$$ 
\left (\int^1_{\sqrt \frac{1-\tau^{p-1}}{1-\tau^{p+3}}} (1-t^2)^\frac{3-p}{p-1} dt 
+ \int_{-1}^{\tau^2 \sqrt \frac{1-\tau^{p-1}}{1-\tau^{p+3}}}
(1-t^2)^\frac{3-p}{p-1} dt\right ) 
$$
By a similar argument we can treat 
$\|\chi^-_{\tau, p, \omega}\|_{L^{2}(\R)}^{2}$
and we can also deduce \eqref{Lp+1}.

\hfill$\Box$

\end{section}

\end{document}